\documentclass[12pt]{amsart}
\usepackage{graphicx}
\usepackage{tikz-cd}
\usepackage{mathtools}
\usepackage{amsfonts}
\usepackage{graphicx}
\usepackage{verbatim}
\usepackage{textcomp}\usepackage{amssymb}
\usepackage{cite}
\usepackage{color}
\usepackage[all]{xy}
\usepackage{graphicx}
\usepackage{color}
\usepackage{hyperref}
\usepackage{soul}
\hypersetup{
    colorlinks,
    citecolor=black,
    filecolor=black,
    linkcolor=black,
    urlcolor=black
}
\usepackage{tikz}
\usepackage[all]{xy}
\usepackage{graphicx}
\usetikzlibrary{backgrounds}
\newtheorem{theorem}{Theorem}[section]
\newtheorem{lemma}[theorem]{Lemma}
\newtheorem{proposition}[theorem]{Proposition}
\newtheorem{corollary}[theorem]{Corollary}
\newtheorem{problem}[theorem]{Problem}   

\allowdisplaybreaks

\newtheorem{conjecture}[theorem]{Conjecture}

\newtheorem{thmconj}[theorem]{Theorem-Conjecture}

\theoremstyle{definition}
\newtheorem{definition}[theorem]{Definition}
\newtheorem{remark}[theorem]{Remark}

 \newenvironment{thm:mingrad}
 {{\sc Proof of Theorem~\ref{thm:mingrad}.}}
{{\sc q.e.d.} \\}
\newenvironment{transmeasure}
 {{\sc Proof of Theorem~\ref{transmeasure}.}}
{{\sc q.e.d.} \\}
 \newenvironment{thm:supptmeasure} 
{{\sc Proof of Theorem~\ref{thm:supptmeasure}.}}
{{\sc q.e.d.} \\}
\newenvironment{straightline} 
{{\sc Proof of Theorem~\ref{straightline}.}}
{{\sc q.e.d.} \\}

{{\sc Proof of Lemma~\ref{federeroutermeasure}.}}%
{{\sc q.e.d.} \\}
{{\sc Proof of Lemma~\ref{cd2general}.}}%
{{\sc q.e.d.} \\}
{{\sc Proof of Lemma~\ref{usordv}.}}%
{{\sc q.e.d.} \\}
{{\sc Proof of Corollary~\ref{corolgeodesic}.}}%
{{\sc q.e.d.} \\}
{{\sc Proof of Theorem~\ref{thmgeodesic}.}}%
{{\sc q.e.d.} \\}
{{\sc Proof of Theorem~\ref{hmco2}.}}%
{{\sc q.e.d.} \\}
{{\sc Proof of claim~(\ref{claim:bdvar}).}}%
{{\sc q.e.d.} \\}
{{\sc Proof of Theorem~\ref{conversetomeasure}.}}%
{{\sc q.e.d.} \\}
\numberwithin{equation}{section}
\newtheorem*{theorem*}{Theorem}
{{\sc Proof of Lemma~\ref{tri1}.}}%
{{\qed} \\}
{{\sc Proof of Theorem~\ref{regularity}.}}%
{{\qed} \\}
\newenvironment{proof:main}%
{{\sc Proof of Theorem~\ref{theorem:mainexistence}.}}%
{{\qed} \\}
\newenvironment{proof:pluriharmonic}%
    {{\sc Proof of Theorem~\ref{theorem:pluriharmonic}.}}%
  {{\qed} \\}  
   \newenvironment{proofof(iv)}%
    {{\sc Proof of $(iv)$.}}%
  {{\qed} \\}  
\newcommand{\R}{\mathbb R}

\newcommand{\Z}{\mathbb Z}

\begin{document}

\title{ Transverse Measures and  Best Lipschitz and Least Gradient Maps}
\author[Daskalopoulos]{Georgios Daskalopoulos}
\address{Department of Mathematics \\
                 Brown University \\
                 Providence, RI}%02912}
%\email{georgios_daskalopoulos@brown.edu}
\author[Uhlenbeck]{Karen Uhlenbeck}
\address{Department of Mathematics, University of Texas, Austin, TX and Distinguished Visiting Professor, Institute for Advanced Study, Princeton, NJ}
%\email{uhlen@math.utexas.edu }

\thanks{
GD supported in part by NSF DMS-2105226.}
\maketitle

\begin{abstract}
Motivated by  work of Thurston on defining a version of Teichm\"uller theory based on best Lipschitz maps between surfaces, we study infinity-harmonic  maps from a hyperbolic manifold to the circle. The best Lipschitz constant is taken on a geodesic lamination.  Moreover, in the surface case the dual problem leads to a function of least gradient  which defines a transverse measure on the lamination.  We also discuss  the construction of least gradient functions from transverse measures via primitives to Ruelle-Sullivan currents.
\end{abstract}

\section{introduction}
%This is the first in a series of papers, where we attempt to connect three different concepts: Best Lipschitz (or $\infty$-harmonic functions), maps of least gradient and measured laminations. The original motivation is to try to understand the work of Thurston \cite{thurston} and eventually develop the analytic techniques 
%This paper fails short of this goal. The analytic subtleties in deal with infinity harmonic maps between surfaces 
%This paper is inspired by the work of Thurston \cite{thurston} on the connection between measured laminations and best Lipschitz maps.

Bill Thurston  in a 1986 preprint which was later revised in 1998 (cf. \cite{thurston})  introduced best Lipschitz maps as a tool of studying an $L^\infty$-version of Teichm\"uller theory. His motivation was to replace the Teichm\"uller distance between two conformal structures $\rho$ and $\sigma$  defined as the dilatation of the Teichm\"uller map by the best Lipchitz constant of a map in the homotopy class of the identity between the corresponding hyperbolic surfaces.  This is what is now known as {\it Thurston's asymmetric metric}. This theory has been further developed by the Thurston school, see for example \cite{papa} or \cite{kassel} and the references therein. We were drawn to  the subject of this paper in part after discussions with Athanase Papadopoulos by the possibility of developing some analytic understanding of this theory. However, even the most basic tools in partial differential equations for approaching this subject are lacking. 

As a starting point, we looked at the subject of this paper: Best Lipschitz maps $u: M^2 \rightarrow S^1$ (or more generally functions  $\tilde u: \tilde M^2 \rightarrow \R$ equivariant under a cohomology class $\rho: \pi_1(M) \rightarrow \R$). This problem does not seem to have been treated in the topological literature (however see Problem~\ref{toyprobm} in the last section on how it fits with Thurston), but our results fit in nicely with existing concepts. In analysis, the subject of $\infty$-harmonic maps, albeit for maps $u: \Omega \rightarrow \R$ for $\Omega$ a domain in $\R^2$ (or even $\R^n$), are in place. See for example: \cite{crandal}, \cite{evans-savin} and \cite{evans-smart}. 

On the opposite side of $\infty$-harmonic maps lies the theory of 1-harmonic maps, which are also known as maps of least gradient. Here, the flavor of the analysis is completely different, see for example \cite{ziemer}, \cite{ziemer2} and \cite{mazon} among other references. See also \cite{juutinen} for a variational construction of least gradient functions for Euclidean domains obtained as limits of $q$-harmonic functions where $q \rightarrow 1$. 

 Thurston  conjectured  that the duality between best Lipschitz maps and measures  should fit well in his theory of geodesic laminations, as stated in the introduction of his paper \cite{thurston}:

{\it I currently think that a characterization of minimal stretch maps should be possible in a considerably more general context (in particular, to include some version for all Riemannian surfaces), and it should be feasible with a simpler proof based more on general principles-in particular, the max flow min cut principle, convexity, and $L^0 \leftrightarrow L^\infty$ duality.}

The goal of this paper is to exhibit the duality proposed by Thurston between best Lipschitz maps ($L^\infty$) and Radon measures ($L^0$) explicitly in the simpler case of functions  and how it fits with the theory of measured laminations and transverse cocycles developed by Thurston and Bonahon. See for example \cite{thurston2}, \cite{bonahon1} and \cite{bonahon2}. In fact,  we will exhibit  the duality  explicitly between the $\infty$-harmonic map $u$ and its dual least gradient map $v \in BV$ inducing the Radon measure $dv$: {\it The map $u$ defines the geodesic lamination, whereas $dv$ the transverse measure on the geodesic lamination defined by $u$.} The duality is more or less given by Hodge duality, namely $*du=dv$ as it will be made precise in this paper.

Following other authors, we first study the limits of the critical points $u_p$ of the functional
\[
u \mapsto \int_M |du|^p *1.
\] 
The $L^\infty$ norms of $du_p$ are uniformly bounded, and we obtain a set of weak limits $u$ of the $\infty$-harmonic equation. The function $u$ is Lipschitz and, as proven by Evans, Savin and Smart in \cite{evans-savin} and \cite{evans-smart}, for Euclidean metrics $u$ is differentiable. However, if $n>2$ it is not known that $du$ is continuous,  and even in the case $n=2$,
$du$ is only H\"older continuous. Since the above results are not available  for non-flat metrics we 
bypass this problem by developing the bare minimum of the theory of comparison with cones for the hyperbolic metric. In particular, we show that, as in the case of Euclidean domains,  there exists a notion of  gradient flow in this setting without assuming any differentiability. 
  
Following Thurston, in order to connect  with topology, in (\ref{normlength}) we introduce the invariant
 \begin{equation}\label{defKref}
  K= \sup  \frac{|<du,\gamma>|}{length(\gamma)}
\end{equation} 
 where  the $\sup$ is over  the set of free homotopy classes of simple closed curves  $\gamma$ in $M$.
Then, we are able to show:
 \begin{theorem}\label{Theorem 1.1} For $M$ compact, hyperbolic in any dimension, and $u$ an $\infty$-harmonic map constructed as a limit of $u_p$,  the local Lipschitz constant $L_u$ of $u$ (see Section~\ref{lipconst}) takes on its maximum $L$ on a geodesic lamination $\lambda_u$. Moreover, $ K$ is equal $L$, the Lipschitz constant of $u$.
 \end{theorem}

This is a combination of Theorems~\ref{straightline} and~\ref{K=L}.
We then turn to the dual problem in dimension 2.  A  duality of this form has been noticed before by Aronson and Lindqvist as far back as 1988 (cf. \cite{aronsonlin}), but it may go even further  back to Werner Fenchel in 1949 (cf. \cite[p.81-82]{temam} and \cite{fenchel}).  From the Euler-Lagrange equations for $u_p$, we know $|du_p|^{p-2}*du_p$ is closed, and we normalize it and set it equal to $dv_q$, where $v_q$ lifts to a function with factors of automorphy varying with $q$. Here $1/p + 1/q = 1$. The function $v_q$ is a critical point of $\int_M |dv|^q*1$ where now $q \rightarrow 1$. This leads to one of our main theorems:
\begin{theorem}\label{Theorem 1.2}
 The set of weak limits $v_q \rightarrow v$ as $q \rightarrow 1$ is nonempty. A limit $v$ is  of least gradient  among maps defining the same homology class. The support of $dv$ is on the  lamination $\lambda_u$  on which $L_u = L$,  obtained from any dual $\infty$-harmonic map $u$.
\end{theorem}
This is a combination of Theorem~\ref{lemma:limmeasures2}, Theorem~\ref{thm:supptmeasure} and  Theorem~\ref{thmlegr}.

Note that so far we have restricted ourselves to the case of maps to $S^1$. This was done for the sake of simplicity only and everything can be generalized to arbitrary real cohomology classes $\rho \in H^1(M, \R)$. In other words, we can replace best Lipschitz maps $u: M \rightarrow S^1$ by functions $\tilde u$ defined on the universal cover that are equivariant under a homomorphism $\rho:\pi_1(M) \rightarrow \R$. In Section~\ref{sectevp} we describe this extension of our results to the equivariant case.

Finally, we turn to the concepts in the Thurston literature.  We assume $M = M^2$ is a hyperbolic surface and define transverse measures in Section~\ref{sect7}. The next is one of the main results of the paper:
\begin{theorem} A least gradient map $v$ as in Theorem~\ref{Theorem 1.2}  induces a transverse measure on the naturally oriented geodesic lamination on which $L_u$ takes on the best Lipschitz constant $L$.
\end{theorem}
This is Theorem~\ref{transmeasure}. We show that $v$ is constant on the connected components of the complement of the lamination $\lambda_u$ from which we can construct a transverse cocycle in the sense of \cite{bonahon2}. We use the approximation by $v_q$ to show that it is non-negative and thus defines a transverse measure.  In addition, we discuss the connection between measured laminations and functions of bounded variation.

 More precisely,  we construct a measure $\nu$ on admissible transversals $f:[c,d] \rightarrow M.$  In the universal cover, we have a function of bounded variation $v.$  We show that for an admissible transversal $f$, $g = f^*v$ is a function of bounded variation and we define the transverse measure $\nu(f)$ as the norm of $g = f^*v$ on the interval.  Every function of bounded variation $g$ on an interval can be written as the sum of a non-increasing function $g^+$ and  non-decreasing function $g^-$, and the norm  is simply $|g^+(d) - g^+(c)| + |g^-(d) - g^-(c)|.$  This norm is invariant under homotopy through admissible transversals.  The difficulty is to match the topological definition of transversal and transverse measure with the analytical definition of  bounded variation.

\begin{theorem}\label{bird} A function $\tilde v$ on the universal cover of a hyperbolic surface that is equivariant, locally bounded and constant on the plaques of an oriented lamination defines a transverse cocycle $\nu$. Moreover,  if  $\nu$  is a transverse measure, then $\tilde v$ is locally of bounded variation.
  \end{theorem}

This is proved in Theorem~\ref{thm:tranco} and Theorem~\ref{transcoismeasthm}. 
We then  prove a partial converse to Theorem~\ref{bird}. In a 1975 paper, Ruelle and Sullivan constructed in a very general setting  closed currents from transverse measures. We show that we have enough regularity to make this rigorous in the setting of transverse measures on geodesic laminations of surfaces.
 More precisely:

 \begin{theorem}The Ruelle-Sullivan current associated to an oriented geodesic lamination in a hyperbolic surface is well defined and closed. A primitive $v$ of the  Ruelle-Sullivan current exists and is locally of bounded variation.
 \end{theorem} 
 This is a combination of Theorem~\ref{welldefcur} and Theorem~\ref{conversetomeasure}. We expect to show that for an appropriate choice of orientation $v$ is always a least gradient in a future paper.

 We also point out that the decomposition of measured laminations into minimal components  corresponds to the decomposition of functions of bounded variation
 \[ 
 v = v_{jump} + v_{cantor},
 \]
 where $dv_{jump}$ has support on closed geodesics in the lamination and $dv_{cantor}$ has support on the minimal components with leaves infinite  geodesics.
 We end the paper by giving a  long list of open problems. 
 \vspace{.1 in}
 
A brief outline of the paper is as follows:
\begin{itemize}
\item Section~\ref{sect:pharmmaps}: {\it$p$-harmonic maps (and their limits).}
This is a review of the properties of the $p$-harmonic equation and its limits as $p \rightarrow \infty$. We also prove a useful maximum estimate needed in Section~\ref{consmeasure}.
\item Section~\ref{sect:conjug}: {\it The conjugate equation for finite $q$.} We define the dual harmonic map for $1/p+1/q=1$ and introduce the adapted coordinate system. 
We also discuss the flat structure induced by the coordinate $(u_p,v_q)$.
\item Section~\ref{qgoesto1}: {\it The limit $q \rightarrow 1$.}
The limiting map of bounded variation is constructed.
\item Section~\ref{sect:crandal}: {\it Geodesic laminations associated to the $\infty$-harmonic map.}
In this section, $M$ is hyperbolic of any dimension. We study comparison with cones and provide a proof of Theorem~\ref{Theorem 1.1}.
\item Section~\ref{consmeasure}:  {\it The concentration of the measure.}
A straightforward but surprising application of the Euler-Lagrange equations for $u_p$ and $u$. The statement is roughly that  small $L^1$-norm implies that the dual measure $dv$ has support on the lamination and is in a weak sense  orthogonal to $du$. We are able to apply this to properties of $v$, for example to show that $v$ is of least gradient.
\item Section~\ref{sect7}: {\it Construction of the transverse measure from the least gradient map.}
This is a tricky section, as it necessitates forming a bridge between the concepts in analysis and the concepts in topology. To evaluate a measure $dv$ on a curve $\gamma$, analysis usually requires the derivative of $\gamma$ to exist, whereas it is important in topology to define the measure on continuous transversals.
\item Section~\ref{ruelles}: {\it From transverse measures to functions of bounded variation.}
 We construct the Ruelle-Sullivan current and show that we have enough regularity to make this rigorous in the setting of transverse measures on laminations. We also construct a primitive to the Ruelle-Sullivan current and  discuss the role of BV functions on  transverse measures.
\item Section~\ref{conjectures}: {\it Conjectures and open problems.}
We give a list of some problems we think we can solve given enough time. The last few problems are enticing. Where there is some analysis, there is little topology, and vice versa.
\end{itemize}
{\bf Acknowledgements.}
Many thanks to those who helped us untangle both the analysis and the topology. Special thanks goes to Craig Evans, Camillo de Lellis, Athanase Papadopoulos, Rafael Poitre and Ovidiu Savin for useful conversations. We have enjoyed working on this project and hope others will appreciate it as well.

\section{$p$-harmonic maps}\label{sect:pharmmaps}
In this section we collect basic facts about different types of harmonic functions. We review the notion of $p$-harmonic functions both for finite $p$ and $p=\infty$. Solutions for finite $p$ obey the theory of elliptic differential equations (cf. \cite{uhlen}), whereas  solutions  to the $\infty$-Laplacian are constructed as limits of  harmonic functions for finite $p$. For the Euclidean domain metric, the local theory of the $\infty$-Laplacian is well-known to analysts. See for example, \cite{arcrju}, \cite{crandal}, \cite{jensen}, \cite{evans-savin}, \cite{evans-smart} and \cite{lindqvist} and all the references therein. The complication in our situation comes from the fact that the maps we are considering take values in $S^1$ instead of $\R$ and also that the domain metric is non-Euclidean. 
%In fact, we will have to work out from scratch  the results we need about infinity harmonic functions from hyperbolic manifolds and avoiding to use the theory that so far has only been developed for Euclidean domains. 

\subsection{$p$-harmonic maps to the circle} Let $(M, g)$ be a closed  smooth Riemannian manifold  of dimension $n \geq 2$ and let $\tilde M$ denote its universal cover with the induced Riemannian metric. By a {\it {  fibration of $M$ over the circle}} we mean a non trivial homotopy class of maps   
\[
f: M \rightarrow S^1.
\]
Note that we are not assuming apriori that $f$ is a submersion.
Equivalently, and for
\[
\rho=f_*:\pi_1(M) \rightarrow \Z=\pi_1(S^1)
\]
we can consider instead the class of $\rho$-equivariant maps
\[
\tilde f: \tilde M \rightarrow \R,
\]
i.e maps satisfying
\begin{equation*}\label{equiccmd}
\tilde f(\gamma \tilde x)=  \tilde f( \tilde x)+ \rho(\gamma), \ \forall \gamma \in \pi_1(M)\  \mbox{and} \   \forall \tilde x \in \tilde M.
\end{equation*}
On the space 
 $W^{1,p}(M, S^1)$, $n<p< \infty$ of maps in a fixed  homotopy class 
% defined by $\rho$ as the space of $W^{1,p}_{loc}$-functions on the universal cover satisfying 
% (\ref{equiccmd}) almost everywhere. Since $|df|$ is invariant under $\pi_1(M)$
 we consider the functional
\begin{equation}\label{pharmfun}
J_p(f)=\int_M|df|^p *1.
\end{equation}
 A unique minimizer $u_p$ of the functional $J_p$ exists in the homotopy class and is 
 called a {\it{$p$-harmonic map.}} It satisfies the equation
\begin{equation}\label{pharm}
div(|\nabla u_p|^{p-2}\nabla u_p)=d^*(|du_p|^{p-2}du_p)=0.
\end{equation}

The existence of $u_p$ is  standard.  Consider a minimizing sequence $u^j$ of $J_p$ in a homotopy class. This makes sense since $p>n$ and hence the maps are continuous. By weak compactness and lower semicontinuity  $u^j$  converge weakly in $W^{1,p}$ to a minimizer
$u_p$ which is in the same homotopy class. 
The argument above can be modified also in the general case $ p>1  $ by minimizing $J_p$  on the space of Lipschitz maps in the given homotopy class.  Since we are only interested in large values of $p$ we  omit the details.

%We can view $u^j$ as equivariant functions $\tilde u^j$ on the universal cover satisfying (\ref{equiccmd}). Fix a smooth fundamental domain $\digamma \subset \tilde M$ for the action of $\pi_1(M)$.  Given $W \subset \subset \tilde M$, we can choose by compactness a finite set of $\gamma \in \pi_1(M)$ such that $W \subset \cup_\gamma \gamma \digamma$. We have a uniform bound of te $L^p$ norm of $du^j$ on $W$.

%By replacing   $\tilde u^j$ by $\tilde u^j-c^j$ where $c^j$ is the average value of $\tilde u^j$ on $W$ and applying the Poincare inequality, we may assume that $\tilde u^j$  are bounded in $W^{1,p}$. Hence there is subsequence converging weakly to a function on $W^{1,p}(W)$minimizer $\tilde u_p$. By a diagonalization argument we can prove that there exists a map $\tilde u_p$ defined on the universal cover such that 
%$\tilde u^j$  converges to $\tilde u_p$ in $W^{1,p}_{loc}$. The pointwise convergence of $\tilde u^j$ to $\tilde u_p$ implies that $u_p$ satisfies the same equivariance property (\ref{equiccmd}) almost everywhere. By restricting to the fundamental domain $\digamma$ and applying semicintinuity of energy we obtain that $\tilde u_p$ is a minimizer.

There is an abundance of literature on regularity of $p$-harmonic functions and $p$-harmonic maps.
%\begin{proposition}[cf. \cite{aless}, Proposition 3.3] Assume $n=2$. There exist constants $\alpha$, $k$, $0<\alpha \leq 1$, $0\leq k <1$
% depending only  on $p$, and $g$ such that for any coordinate neighborhood $\Omega$ there exist and $s, h \in  C^\alpha(\Omega)$ such that 
% \[
% (u_p)_z = e^sh, 
% \]
% where $h$ is a $k$-quasiregular mapping in $\Omega$.
%\end{proposition}
 
% By  \cite{aless}, Proposition 2.5 the real and imaginary part of $h$ have isolated $k$-prong singularities with $k>0$. By an Euler characteristic argument we conclude that there are finitely many singular points. Thus,

\begin{theorem} \label{nonsingular} Let
$u_p: (M, g) \rightarrow S^1$
denote the $p$-harmonic map in the homotopy class of $f: M \rightarrow S^1$. Then,  $u_p \in C^{1, \alpha}$ and if $\Omega \subset \subset \{|du_p| \neq 0\}$, then $u \in  C^\infty(\Omega)$ (cf. \cite{uhlen}). If $n=2$, then  the number of singular points $|du_p|= 0$ is finite and bounded by the Euler characteristic of $M$ (cf. \cite{manfredi} and \cite{aless}).  
\end{theorem}

\subsection{ Best Lipschitz maps and $\infty$-harmonic maps}\label{lipconst}
For K  a subset of a Riemannian manifold  $(M, g)$,  and $f: K \rightarrow S^1$,  its   {\it{Lipschitz constant}} in $K$  is defined by
\[
L_f(K):= \inf \{ L \in \R : d_{S^1}(f(x), f(y)) \leq Ld_g(x,y) \  \forall x, y \in K \} .
\]
In the above, $\inf \emptyset = +\infty$. We say $f$ is Lipschitz in $K$, if 
$L_f(K) < +\infty$. We write 
\[
L_f=L_f(M)
\]
for the {\it global Lipschitz constant.}
For $U$ be an open subset of $M$ and $x \in U$, we define the {\it{local Lipschitz constant }}
\[
L_f(x) := \lim_{r \rightarrow 0} L_f(B_r (x)).
\]
Clearly, if $f$ has Lipschitz constant $L$ in $U$, then
$L_f(x) \leq L.$

\begin{proposition}\label{crandal0}[\cite{crandal}, Lemma 4.3] For any function $f:U \rightarrow \R$, 
\begin{itemize}
\item $(i)$ the map $x \mapsto L_f(x)$ is upper semicontinuous.
\item $(ii)$  $ df \in  L^\infty(U)$ holds in the sense of distributions if and only if $L_f(x)$ is bounded
on $U$;  then
\[
\sup_{x \in U} L_f(x)= |df|_{L^\infty(U)}\ \mbox{and} \ L_f(x)=\lim_{r \rightarrow 0}|df|_{L^\infty(B_r(x))}.
\]
\item $(iii)$ If $U$ is convex, then
\[
 L_f(U)=|df|_{L^\infty(U)}.
\]
\end{itemize}
\end{proposition}

%In this subsection $(M, g)$ continues to denote a closed Riemannian manifold of dimension $n>1$  and  
%$\tilde M \simeq \R^n$ denotes its universal cover with the metric induced from $g$.
\begin{definition}\label{defbl} 
The infimum of the global Lipschitz constant $L_f$ for all $f:M \rightarrow S^1$ in a fixed homotopy class is called the {\it{best Lipschitz constant}}. A Lipschitz map
\[
u: M \rightarrow S^1
\]
is called {\it{a best Lipschitz  map}},  if for any  Lipschitz map $f: M \rightarrow S^1 $  homotopic to $u$ 
\begin{equation*}
L_u \leq L_f. 
\end{equation*}

\end{definition}

%\begin{definition}\label{infinharm} A Lipschitz map
%\[
%u: M \rightarrow S^1.
%\]
%is called $\infty$-harmonic (resp. $\infty$-subharmonic or $\infty$-superharmonic) if its lift to the universal cover
%\[
%\tilde u: \tilde M \rightarrow \R
%\]
%is  an $\infty$-harmonic (resp. $\infty$-subharmonic or $\infty$-superharmonic) function. By this we mean a viscosity solution (resp. subsolution or supersolution) of the 
%\emph{$\infty$-Laplace equation}
%\[
%\triangle^\infty \tilde u= \frac{1}{2} <grad(\tilde u),  grad(|grad(\tilde u)|^2)>_g= 0, 
%\]
%where $<,>_g$ denotes the inner product associated with $g$. For more details on the notion of viscosity solutions we refer to \cite{crandal} or \cite{lindqvist}.
%\end{definition}
%
%In this paper we will only deal with a special type of viscosity solutions that are also variational solutions to the $\infty$-Laplace equation in a fixed homotopy class. The basic construction  is summarized in the following

\begin{theorem}\label{thm:limminimizer}
Let $(M, g)$ be a closed  smooth Riemannian manifold  of dimension $n \geq 2$.  
For each $p>n$, let $u_p$ be the $p$-harmonic map homotopic to  
 a Lipschitz map $f: M \rightarrow S^1$. Given a sequence $p \rightarrow \infty$, there exists a subsequence (denoted again by $p$)  and a  Lipschitz map $u: M \rightarrow S^1$ such that:
\begin{itemize}
\item $(i)$ $u_p \rightarrow u$ uniformly.
\item $(ii)$ $u$ is  best Lipschitz with Lipschitz constant equal to the best Lipschitz constant. Furthermore, $u$ also minimizes the Lipschitz constant for the local  Dirichlet problem subject to its own boundary conditions.
\item $(iii)$ $du_p  \rightharpoonup du \ \mbox{and} \ *du_p \rightharpoonup *du \ \mbox{weakly in} \  L^s \ \forall s>n$.
\item  $(iv)$ $du$ is closed. 
\end{itemize}
\end{theorem}
\begin{proof} We follow directly the proof of the existence of $\infty$-harmonic functions for the Dirichlet problem (cf. \cite{lindqvist}, Chapter 3).  We only give a sketch:
Take a sequence $p\rightarrow \infty$ and $\epsilon >0$. By H\"older, and the fact that $u_p$ is a minimizer of $J_p$, we have for $n<s<p$ large,   
\begin{eqnarray*}
\frac{1}{vol(M)^{1/s}}J_s(u_p)^{1/s} &\leq& \frac{1}{vol(M)^{1/p}} J_p(u_p)^{1/p} \\
&\leq& \frac{1}{vol(M)^{1/p}} J_p(f)^{1/p} \\
&\leq&  |df|_{L^\infty}.
\end{eqnarray*}
 Hence, $|du_p|_{L^s}$ is uniformly bounded and thus, after passing to a subsequence 
(denoted again by $p$),
\[
du_p \rightharpoonup du \ \mbox{weakly in} \  L^s.
\]
By semicontinuity,
\[
 |du|_{L^s} \leq \liminf  |du_p |_{L^s} \leq |df|_{L^\infty}.
\]
By a diagonalization argument, we can choose a single subsequence $p$ such that
\[
du_p \rightharpoonup du  \ \mbox{weakly in}  \ L^s,\ \forall s
\]
 and by taking $s \rightarrow \infty$,
\[
|du |_{L^{\infty}} \leq |df|_{L^\infty}.
\]
By going to the universal cover, the same inequality holds for $\tilde u$ and any $\tilde f$. Thus, by the convexity of $\tilde M$ and
 the mean value theorem (cf. Proposition~\ref{crandal0}, $(iii)$), it follows that $u$ is a best Lipschitz map with best Lipschitz constant $L_u=L$. 

Moreover, as in \cite{lindqvist}, Theorem 3.2, $u$ is also a local minimizer for the Dirichlet problem subject to its own boundary conditions.  
Properties $(iii)$-$(iv)$ follow immediately from the argument sketched above.
%To show that $du$ is unique, suppose $u$ and $u'$ are two best Lipschitz maps obtained as weak sequential limits of $u_p$. Lift to the universal cover, to $\tilde u$ and $\tilde u'$. In a fundamental domain
%\[
%\max \tilde u' - \min \tilde u = m
%\]
%and because $\tilde u$ and $ \tilde u'$ are equivariant with respect to the same representation, this is true on all of $\tilde M$. Hence
%\[ 
%\max \tilde u' \leq \tilde u +m \ \mbox{on} \ \tilde M.
%\]  
% By continuity there is some $m'\leq m$ such that $\tilde u' - \tilde u\leq m'$ and some non-empty set of points $x$ such that $\tilde u'(x) - \tilde u(x)=m'$. But $\tilde u'-\tilde u$ descends to an infinity harmonic function on $M$. This violates the maximum principle for infinity harmonic functions (cf. \cite{lindqvist}, Proposition 6.2), unless 
%$\tilde u'(x) = \tilde u(x) + m'$ for all $x$. Hence $du = du'$.

\end{proof}

\begin{definition}\label{infinharm}
We call $u$ as in the previous theorem {\it{$\infty$-harmonic}}. Notice that in this paper we are only concerned with solutions that are limits of $p$-harmonic maps to $S^!$. Sometimes these are called {\it{variational solutions of the $\infty$-Laplace equation}}. If the domain is Euclidean then variational solutions are also {\it{viscosity solutions}} of the $\infty$-Laplace equation. We will not attempt to develop such a notion for non-Euclidean metrics in the present paper. For more details on the notion of viscosity solutions in Euclidean space we refer to \cite{crandal} or \cite{lindqvist}. For open problems we ask the reader to go to the last section.

\end{definition}

\begin{remark}If the domain metric is Euclidean it has been shown that $u$ has the additional properties
\begin{itemize}
\item  If $n=2$, then $du$ and $*du$ are in $C^\alpha$ (cf. \cite{evans-savin}).
\item  If $n>2$, then $du$  exists everywhere but is not known to be continuous (cf. \cite{evans-smart}).
\end{itemize}
It is very likely that these results also hold for the hyperbolic metric but since they only have been written  down carefully for the Euclidean metric  we will not use them in this paper.
\end{remark}
 
 \begin{lemma}\label{pintconvto}
\[
\lim_{p \rightarrow \infty} \left( \int_M |du_p|^p*1 \right )^{1/p} =  L.
\]
\end{lemma}
\begin{proof}If $f$ denotes a  best Lipschitz map in the homotopy class, the fact that $u_p$ is a minimizer for $J_p$ implies
\[
\left( \int_M |du_p|^p*1 \right )^{1/p}\leq \left( \int_M |df|^p*1 \right )^{1/p} \leq (vol M)^{1/p} L.
\]
Hence the $\limsup$ is less than equal to $L$. On the other hand, if $\liminf=a<L$, then proceeding as in the proof of the previous theorem, there exists a Lipschitz map $u$ such that
\[
|du|_{L^\infty}\leq a<L
\]
which contradicts the best Lipschitz constant.
\end{proof}

\subsection{The Maximum estimate}

We know that the $p$-harmonic maps $u_p$ are smooth away from their critical points.  However, in Section~\ref{consmeasure} we will need the following result:

\begin{proposition}\label{maxest}  $\lim_{p \rightarrow \infty} \max |du_p| = L.$
\end{proposition}

Since $L$ is the best Lipschitz constant $\max |du_p| \geq L$, so we need prove an upper bound.
Let $s = u_p/L.$   This simplifies the normalizations.  The size of $S^1$ does not enter into the calculations. %and
%\[
% ||w||_q = \left(\int_M |w|^q *1 \right)^{1/q}.
%\]

\begin{lemma} Let 
\[
w = |ds|^p = (|du_p|/L)^p.
\] 
Let $W^{1,2} (M) \subset L^{2a}(M)$, where $a$ is arbitrary for $dim M = 2$ and $n/(n-2)$ when $dim M > 2$.  Then
\[
     \max w = \lim_{l \rightarrow \infty} |w|_{L^l}  \leq Cp^{a/a-1}. 
\]
The constant $C$ depends only on the norm of the embedding and the Ricci curvature of $M$ and not on $p.$
\end{lemma}

\begin{proof}
The Proposition follows easily from the lemma, as
\[ 
\max |du_p| = \max w^{1/p}L \leq (Cp^{a/a-1})^{1/p} L \rightarrow L.
\]
The proof of the lemma is standard, using the Bochner formula and Moser iteration and only needs to be included to keep track of $p.$ So we will be brief. 

In the usual way, we integrate the Euler-Lagrange equations $d^*|ds|^{p-2}du = 0$ against a term $d^*\phi ds$ where $\phi$ is a non-negative function on $M.$ We integrate by parts and complete the Laplacian to obtain
\begin{eqnarray*}
\int_M<\triangle (|ds|^{p-2}ds), \phi ds> *1&=&\int_M<d |ds|^{p-2} \wedge ds, d\phi \wedge ds>*1\\
&=&\int_M<d |ds|^{p-2}, d\phi> |ds|^2-<d |ds|^{p-2}, ds> <d\phi, ds>*1.
\end{eqnarray*}
Next use the Bochner formula for 1-forms and integrate by parts to obtain
\begin{eqnarray*}\label{myst1}
\int_M<\nabla(|ds|^{p-2}ds), \nabla(\phi ds)>*1&=&-\int_M Ricc(ds,ds) \phi |ds|^{p-2}*1 \nonumber\\
&+&\int_M<d |ds|^{p-2}, d\phi>  |ds|^2-<d |ds|^{p-2}, ds> <d\phi, ds>*1.
\end{eqnarray*}
\\
After expanding the left hand side and bringing the second term to the left hand side, we obtain an expression
%\begin{eqnarray}\label{myst1}
%-\int_M Ricc(ds,ds)|ds|^{p-2}\phi*1=\int_M <D_2(|ds|^{p-2}d_1s),D_1(\phi d_2s)>*1
% \end{eqnarray} 
%%-------------------------------------------
%%
%%The proof of the Lemma is standard, using the Bochner formula and Moser iteration and only needs to be included to keep track of $p.$ So we will be brief. 
%%In the usual way, we integrate the Euler-Lagrange equations $d^*|ds|^{p-2}du = 0$ against a term $d^*\phi ds$ where $\phi$ is a non-negative function on M. We integrate by parts, exchange $\{D,D\}$ for a term involving Ricci curvature, and integrate back. \\
%% -----------
%% 
%%Alternatively, we can 
%%%start with the Bochner formula for 1-forms $dd^*+d^*d=\nabla^*\nabla+ Ricc$ and integrate $|ds|^{p-2}ds $ against a term $\phi ds $. Note that in both cases 
%We use the fact that $|ds|^{p-2}ds  \in W^{1,2}$.
%
%
%
%
%When  worked out we obtain an expression
%\begin{eqnarray}\label{myst1}
%-\int_M Ricc(ds,ds)|ds|^{p-2}\phi*1=\int_M <D_2(|ds|^{p-2}d_1s),D_1(\phi d_2s)>*1
% \end{eqnarray}
% Here the numbers 1 and 2 refer to the pairing in the inner product.\\
% -----------
%Integrate by parts, using the fact that $|ds|^{p-2}ds  \in W^{1,2}$ and the Euler-Lagrange equations
%\[ 
%d*|ds|^{p-2}ds = 0,
%\]
%to obtain an expression
 \begin{eqnarray}\label{myst11}
 \int_M A*1 &=&- \int_M  Ricc(ds,ds)\phi |ds|^{p-2}*1.
 \end{eqnarray} 
% In order to assure that the integrals above converge we use the fact that by \cite[Proposition 2]{ivan} $|ds|^{(p-2)/2}ds \in W^{1,2}$. By taking $d$ of the above expression 
% $d(|ds|^{(p-2)/2})ds \in L^2$ and hence also $\nabla(|ds|^{(p-2)/2}ds)-d(|ds|^{(p-2)/2})ds=|ds|^{(p-2)/2}\nabla ds \in L^2$.
% From this and the fact that $ds \in C^\alpha$ we can deduce that $|ds|^rds \in W^{1,2}$ for all $r \geq (p-2)/2$. 
% This justifies that the integral in (\ref{myst11})  converges.
% 
% We now proceed with the calculation.
%The integrand on the left-hand side has four terms when worked out
%\begin{eqnarray}
%A &=&  \phi |ds|^{p-2}((|\nabla ds|^2 +(p-2)|d|ds|)^2 + |ds|^{p-1}(d|ds|,d\phi) \nonumber\\
%   &+& <d|ds|^{p-2}ds, d\phi ds>.
%\end{eqnarray}
%\begin{eqnarray*}
%\lefteqn{ \phi (|ds|^{p-2}|\nabla ds|^2 +<d |ds|^{p-2}ds, \nabla ds>)} \nonumber\\
%   &+& |ds|^{p-2}<\nabla ds,d\phi ds> + <d|ds|^{p-2}ds, d\phi ds>\nonumber\\
%   &=& \phi |ds|^{p-2}(|\nabla ds|^2 +(p-2)(d|ds|)^2)  \\
%   &+& |ds|^{p-2}<\nabla ds,d\phi ds>  + <d|ds|^{p-2}ds, d\phi ds>\nonumber.
%\end{eqnarray*}
%\begin{eqnarray*}
%\lefteqn{ \phi (|ds|^{p-2}|\nabla ds|^2 +<d |ds|^{p-2}ds, \nabla ds>)} \nonumber\\
%   &+& |ds|^{p-2}<\nabla ds,d\phi ds> + <d|ds|^{p-2}ds, d\phi ds>\nonumber\\
%   &=& \phi |ds|^{p-2}(|\nabla ds|^2 +(p-2)(d|ds|)^2)  \\
%   &+& |ds|^{p-2}<\nabla ds,d\phi ds>  + <d|ds|^{p-2}ds, d\phi ds>\nonumber.
%\end{eqnarray*}
%\\
%-------
The integrand on the left-hand side has four terms when worked out.
\begin{eqnarray*}
A &=& \phi |ds|^{p-2}(|Dds|^2 +(p-2)|d|ds||^2) \\
&+& |ds|^{p-1}<d|ds|,d\phi>+ <d|ds|^{p-2},ds><d\phi,ds>.
\end{eqnarray*}
Recall $w = |ds|^p.$ We insert $\phi = w^{2l-1}, l\geq1/2$ in the expression.  Note
\[
dw = p|ds|^{p-1}d|ds|, \ d\phi = p(2l-1)|ds|^{2pl-p-1}d|ds|.
\]
and that all four terms in  $A$ are non-negative. We ignore the last term.
Using  the inequality $|d|ds|| \leq |D ds|$ we obtain,
\begin{eqnarray}\label{expA}
A &\geq&  |ds|^{p(2l-1)} |ds|^{p-2}(|D ds|^2 +(p-2)|d|ds||^2)  \nonumber\\
&+& |ds|^{p-1}<d|ds|,p(2l-1)|ds|^{2pl-p-1}d|ds|> \\
   &\geq&(p-1+p(2l-1))|ds|^{2pl-2}|d|ds||^2.  \nonumber
   \end{eqnarray}
   Note
\begin{eqnarray}\label{expw}
1/pl |dw^l|^2 &=&1/pl|lw^{l-1}dw|^2  \nonumber\\
&=&1/pl (l |ds|^{p(l-1)} p|ds|^{p-1}d|ds|)^2\\
&=& pl |ds|^{2pl-2}|d|ds||^2.  \nonumber
   \end{eqnarray}
 Using (\ref{expA}) and (\ref{expw})
 \begin{eqnarray}\label{myst3}
 \frac{1}{pl} |dw^l|^2 \leq \frac{(p-1) +(2l-1)p}{(pl)^2} |dw^l|^2 \leq A.
\end{eqnarray}
Also the right-hand side of (\ref{myst11}) is bounded
by
 \begin{eqnarray}\label{myst345}
 - \int_M  Ricc(ds,ds)\phi |ds|^{p-2}*1 \leq R |w^l|_{L^2}^2
 \end{eqnarray}
 and $R$ is the maximum of the negative Ricci curvature.
Combining (\ref{myst1}), (\ref{expA}), (\ref{myst3}) and (\ref{myst345}) with the Sobolev embedding theorem we get
\begin{eqnarray*}\label{myst4}
|w|_{L^{2la}}^{2l} &\leq& \gamma (|dw^l|_{L^2}^2 + |w^l|_{L^2}^2  )\\
&\leq& \gamma (pl R + 1) |w^l|_{L^2}^2 .
\end{eqnarray*}
Here $\gamma$ refers comes from norm of the Sobolev embedding.  We next simple take $1/2l$-root of this inequality to get
\[
      |w|_{L^{2la}} \leq (Cpl)^{1/2l} |w|_{L^{2l}}.
\]
Now let $l_0 = 1/2$ and $l_{i+1} = a l_i,$ and iterate the inequality. It is an easy exercise to see that 
\[
   |w|_{L^{a^{j+1}}} \leq (Cp)^{\sum_0^j  \frac{1}{a^i}} a^{\sum_0^j   \frac{i}{a^i}} |w|_{L^1}.
\]
The result follows from this.
\end{proof}

\section{The conjugate equation for finite $q$}\label{sect:conjug}
In this section $dim (M)=n=2$. Let $1<q \leq p< \infty$ such that $1/p + 1/q = 1$.
 For each $p$-harmonic map $u_p$, we construct  dual harmonic functions $\tilde v_q$ defined on the universal cover of $M$ and equivariant with respect to representations $\alpha_q: \pi_1(M) \rightarrow \R$. For functions on the plane, this duality has already appeared in  \cite{aronsonlin}. Our main result  is to show that the functions $\tilde v_q$ are locally uniformly  bounded and the representations $\alpha_q$ are uniformly bounded. Together, away from the zeroes of $\tilde u_p$,  the two functions $\tilde u_p$ and $\tilde v_q$ define a convenient coordinate system on the universal cover, called the {\it{adapted coordinate system.}}
\subsection{The conjugate harmonic equation}
Fix $2\leq p < \infty$ and define $0< q \leq 2$ by
\begin{equation}\label{form:conjugate}
\frac{1}{p}+\frac{1}{q}=1.
\end{equation}
Let $u_p$ be a minimizer of the functional 
$J_p$ in the homotopy class of a Lipschitz map $f: M \rightarrow S^1$.
Let 
\[
\tilde u_p: \tilde M \rightarrow \R
\]
be the lift of $u_p$ to the universal cover, equivariant under 
$\rho: \pi_1(M) \rightarrow \Z$. We define the dual 1-form $ \tilde \Psi_q$
\begin{equation}\label{dualform}
\tilde \Psi_q=|d \tilde u_p|^{p-2}*d\tilde u_p.
\end{equation}
\begin{lemma}  $ \tilde \Psi_q$ is a closed, invariant form under the action of $\pi_1(M)$, hence there exists a unique primitive
\[
\tilde w_q: \tilde M \rightarrow \R, \ \ d\tilde w_q= \tilde \Psi_q
\]
equivariant under the period homomorphism
\[
 \beta_q: \pi_1(M) \rightarrow \R; \ \ \beta_q(\gamma)=\int_z^{\gamma z} \tilde \Psi_q
\]
and normalized as
\[
\int_\digamma \tilde w_q *1=0
\]
where $\digamma$ is a fixed fundamental  domain in $\tilde M$.
\end{lemma} 
\begin{proof} The condition $d \tilde \Psi_q=0$ is just the $p$-harmonic equation (\ref{pharm}). The invariance of $\tilde \Psi_q$ follows from the fact that 
$\tilde u_p$ is the pullback of $ u_p$ to the universal cover. For the equivariance under the period homomorphism, see for example \cite{foster}, Section 10.
\end{proof}

\begin{lemma}\label{conjugate} The map $\tilde w_q$  satisfies the $q$-harmonic map equation (\ref{pharm}), for $q$ as in (\ref{form:conjugate}).
\end{lemma}
\begin{proof} 
Notice that  equation (\ref{form:conjugate}) implies
\[
(p-1)(q-2)+p-2=0,
\]
hence
\begin{eqnarray*}
|d\tilde w_q|^{q-2}*d\tilde w_q &=&
 \left| |d\tilde u_p|^{p-2}*d\tilde u_p \right|^{q-2}|d\tilde u_p|^{p-2}*^2d\tilde u_p  \\
&=&  |d\tilde u_p|^{(p-1)(q-2)+p-2} *^2 d\tilde u_p  \\
&=&   -d\tilde u_p.  
\end{eqnarray*}
Thus
\[
d^*(|d\tilde w_q|^{q-2}d\tilde w_q)=0. 
\]
 \end{proof}
 
 \begin{remark}\label{dual-eqn}  Note the duality
\[
 d\tilde w_q=|d\tilde u_p|^{p-2}*d\tilde u_p, \ \ \ -d\tilde u_p=|d\tilde w_q|^{q-2}*d\tilde w_q.
\]
 This can also be explained by means of Fenchel's duality for convex variational integrals. See \cite{fenchel}, \cite{temam} and \cite{aronsonlin} for more details on this kind of analysis.
 Motivated by the case $p=2$, we call $\tilde w_q$ the {\it conjugate harmonic} to $\tilde u_p$. 
 \end{remark}

 \subsection{The normalization} For the rest of the paper we will make the following normalizations:
 
Choose a factor $k_p$ so that
\begin{equation}\label{normintv1}
  \int_M  | k_p d u_p|^p*1 =   k_p.
\end{equation}
Let 
\begin{equation}\label{normintv2}
 U_p = k_p d  u_p  \ \mbox{ and} \ \    V_q = |  U_p|^{p-2}*  U_p.
\end{equation}
 Let $\tilde U_p$ and $\tilde V_q$ denote the lifts to the universal cover and let $\tilde v_q: \tilde M \rightarrow \R$
such that
\begin{equation}\label{normintv3}
 d\tilde v_q=\tilde V_q; \ \ \int_\digamma\tilde v_q *1=0.
 \end{equation}
Notice that $\tilde v_q$ is a rescaling of the conjugate harmonic function $\tilde w_q$ defined in the previous section, $\tilde v_q=k_p^{p-1} \tilde w_q$.
Under the normalizations above, in a fundamental domain $\digamma \subset \tilde M$,
\begin{eqnarray} \label{kappavolform0}
\int_\digamma d\tilde u_p \wedge d \tilde v_q 
&=&
\int_\digamma k_p^{-1} \tilde U_p \wedge | \tilde U_p|^{p-2}* \tilde U_p \nonumber\\
&=&\int_\digamma k_p^{-1} | \tilde U_p|^p *1\\
  &=&1\nonumber.
\end{eqnarray}

\begin{lemma} \label{klemma1}Under the normalizations above, 
$\lim_{p \rightarrow \infty} k_p = L^{-1}. $
%(I conjecture that $k_p$ is about $p^{1/p}$ if we ever need it.) 
\end{lemma}
 \begin{proof} 
 By (\ref{normintv1}) and Lemma~\ref{pintconvto}  
\[ 
\lim_{p \rightarrow \infty} k_p^{(1/p)-1}=\lim_{p \rightarrow \infty}  \left( \int_M |du_p|^p*1 \right )^{1/p} = L.
\]
%  (1-p)/p \ln k_p
By taking logarithms, $ \lim_{p \rightarrow \infty} \ln k_p=-\ln L$,
which implies the Lemma.
\end{proof}

Also, 
\begin{eqnarray} \label{kappanorm0}
 \int_\digamma |d\tilde v_q|*1&=& \int_\digamma | \tilde U_p|^{p-1}*1\nonumber\\
 &\leq &(vol M)^{1/p}  \left( \int_M | \tilde U_p|^p*1\right)^{\frac{p-1}{p}} \ \mbox{(by H\"older)} \\
   &= &(vol M)^{1/p} k_p^{\frac{p-1}{p}} \nonumber\\
   &\leq &(vol M)^{1/p}(L^{-1}+\epsilon_p) \ (\mbox{where $\epsilon_p \rightarrow 0$, \ by Lemma~\ref{klemma1}}).\nonumber\\
   &\approx& L^{-1} \ (\mbox{for $p$ large}).\nonumber
\end{eqnarray}

 We denote by
\[
\alpha_q: \pi_1(M) \rightarrow \R; \ \  \alpha_q(\gamma)=\int_z^{\gamma z}  \tilde V_q
\]
the period homomorphism of the rescaled form $\tilde V_q=d \tilde v_q$. Notice that by definition, $\tilde v_q$ is equivariant under $\alpha$, i.e
\[
\tilde v_q(\gamma z)=\tilde v_q(z)+\alpha_q(\gamma).
\]
It follows that the closed 1-form $\tilde V_q=d \tilde v_q$ is invariant under the action of $\pi_1(M)$ and descends to a closed 1-form $ V_q$ on $M$. 
The representation $\alpha_q: \pi_1(M) \rightarrow \R$ acting on $\R$ via affine isometries, defines a flat  fiber bundle $\tilde M \times_{\alpha_q} \R \rightarrow M$ (specifically a flat affine bundle) and $\tilde v_q$ defines a  section $v_q : M  \rightarrow \tilde M \times_{\alpha_q} \R$. Sometimes it is common to call $v_q$ a {\it{twisted map}}. Under this notation, $V_q=dv_q$.

Note that (\ref{kappavolform0}) and (\ref{kappanorm0}) imply,
\begin{equation} \label{kappavolform}
\int_Md  u_p \wedge dv_q=1
\end{equation}
and
\begin{equation} \label{kappanorm}
 |dv_q|_{L^1(M)} \approx L^{-1}
\end{equation}
for $q$ close to 1. Furthermore,
\begin{equation} \label{alphagamma}
\alpha_q(\gamma)=\int_{\gamma}  dv_q=\int_M \omega_\gamma \wedge dv_q 
\end{equation}
where  
$\omega_\gamma \in \Omega^1(M)$ denotes the closed form Poincare dual to the homology class defined by $\gamma$.
Notice that
for any $\gamma \in \pi_1(M)$ and $0<c_\gamma :=|\omega_\gamma|_{L^\infty(M)}$, we have from (\ref{kappanorm}) for $q$ close to 1
\begin{equation}\label{boundtrlength}
|\alpha_q(\gamma)| \leq c_\gamma \int_M  \left| dv_q \right|*1 \approx c_\gamma L^{-1}.
\end{equation}

 \subsection{The adapted coordinate system} \label{adapted} The pair of functions $(\tilde u_p, \tilde v_q)$ can be used to define a convenient coordinate system on  $\tilde M \backslash \mathcal \{| d\tilde u_p|=0\}$ which we call {\it{the adapted coordinate system}}. More precisely, in the coordinate system $(\tilde u_p, \tilde v_q)$,
the metric $g$ is given by
\begin{equation}
g=\begin{pmatrix} \label{metric}
 \tau_1^2 & 0 \\
 0 & \tau_2^2
\end{pmatrix}
\end{equation}
with
\begin{equation}\label{metric2}
|d \tilde u_p|=\tau_1^{-1}; \ \ \tau_2=  (\tau_1/{k_p})^{p-1}.
\end{equation}
To prove the statement above, it is better to think in terms of the co-metric. Set
\[
\tau_1^{-1}=\left| d\tilde u_p \right |
\]
and note that
\[
\tau_2^{-1}=|d\tilde v_q|=| \tilde U_p|^{p-1}= (\tau_1/{k_p})^{1-p}
\]
and
\[
d\tilde u_p \wedge *d\tilde v_q=  {k_p}^{p-1}d\tilde u_p \wedge |d  \tilde u_p|^{p-2}d \tilde u_p=0.
\]
We have thus proven (\ref{metric}) and (\ref{metric2}). 
For future reference we also choose an orientation on M consistent with the orientation of $d \tilde u_p \wedge d\tilde v_q$. With this orientation $u_p$ is an orientation preserving map to $S^1$ with its standard (counterclockwise) orientation.

\subsection{The normalized flow of $u$}
Let    $\frac{\partial}{\partial u_p}$ denote the vector field dual to the 1-form $du_p$. By Theorem~\ref{nonsingular}, this vector field is globally $C^{\alpha}$ and smooth away from its zeroes. 
 The {\it{normalized gradient flow}} is the flow $\psi_t$  of the vector field $\frac{\partial}{\partial u_p}$. The flow $\psi_t$ lifts to the  universal cover and is given in the 
 local adapted coordinates $(\tilde u_p, \tilde v_q)$,  by
\[ 
(\tilde u_p, \tilde v_q) \mapsto \tilde \psi_t(\tilde u_p, \tilde v_q)=(\tilde u_p+t, \tilde v_q).
\]
Notice that  the  1-forms $d\tilde u_p$ and $d\tilde v_q$ are  invariant under the normalized gradient flow. 

The flow $\psi_t$ on $M$ has interesting dynamics. Choose a regular fiber $\Xi$ and for example assume that $\Xi$ is connected. The normalized gradient flow $\psi_1$ of $u_p$ at time 1 maps $\Xi \backslash$ (points which flow into critical points) to $\Xi \backslash$ (points which flow out of critical points). The map $\psi_1$ gives an interval exchange map of $\Xi$ to itself which is of interest in itself,  though we will exploit it more in this article. See Problem~\ref{Problem 8}.
We next prove:

%$(x,y) \mapsto \psi_t(x,y)=(x+t, y)$ of $u$  (the fibers of $u$ are not preserved by the gradient flow, because the solution $x(t)$ depends on $y$). Furthermore,
%\[
%dy=\frac{\sigma_1^{2-p}}{\kappa}*du=\frac{|du|^{p-2}}{\kappa}*du
%\]
%and
%\[
%dx \wedge dy=\frac{|du|^{p}}{\kappa}*1.
%\]
%\end{lemma}
%\begin{proof} The  statement about invariance follows immediately from Lemma~\ref{gradflow} since
%\[
%\phi_t^*dy=d (y \circ \phi_t)=dy. 
%\]
%For the second notice,
%\[
%dx \wedge *dx  = |dx|^2 \sigma_1\sigma_2 dx \wedge dy =  \kappa \sigma_1^{p-2} dx \wedge dy
%\]
%which implies
%\[
%dy=\frac{\sigma_1^{2-p}}{\kappa}*dx
%\]
%and 
%\[
%dx \wedge dy=\frac{\sigma_1^{2-p}}{\kappa}dx \wedge *dx=\frac{\sigma_1^p}{\kappa}*1.
%\]
%Notice that the statement that dy is closed is equivalent to the $p$-harmonic map equation (\ref{pharm}).
%\end{proof}
\begin{proposition}\label{locbd}The $ \tilde v_q$ are locally uniformly  bounded in $L^\infty$ for all $q$.
\end{proposition}
%\begin{proof} Choose a regular fiber $\Xi$ and assume that $\Xi$ is connected. Then $v_q$ can be chosen locally in the fiber $0\leq v_q \leq 1$ by the choice of normalization. The normalized gradient flow $\psi_1$ of $u_p$ at time 1 maps $\Xi \backslash$ (points which flow into critical points) to $\Xi \backslash$ (points which flow out of critical points). The function $v_q$ is constant on this normalized gradient flow. The map $\psi_1$ gives an interval exchange map of $\Xi$ to itself which is of interest in itself. But we have now given a fundamental domain for $\Xi$ in $\tilde M$ on which $0\leq v_q \leq 1$. Let $m$ be the length of the shortest fiber. If we calculate $v_q(\gamma(t))$ on an arbitrary curve, if $v_q(\gamma(1)) - v_q(\gamma(0) = D$, the curve will have to cross the fiber at least $D-1$ times, and therefore it must be of length at least $(D-1)/m$. This gives a bound on $v_q$ along curves of bounded length.
%The proof is easily modified if there is no connected fibers.
%It is important to note the lack of continuity implied by the bound $(D-1)/m$, not $D/m$.
%\end{proof}

\begin{proof} Choose a regular fiber $\Xi$ of $u_p$ and let $[\Xi] \in H_1(M, \Z)$ denote its homology class. Let $ \omega_\Xi$ be a  closed 1-form representing the Poincare dual of $[\Xi]$ 
%can be represented by the closed form $\omega_\Xi=adu_p$, where $a \in \R$. Indeed, let $\tilde \Xi$ denote the lift of $\Xi$ to the universal cover $\tilde M$ of $M$. 
%Since  $\Xi$ is a regular fiber, it does not contain any critical point of $du_p$, hence, with respect to the adapted coordinates,  we can write in a neighborhood of $\tilde \Xi$,
%\[
%\sigma^*(\omega_\Xi)=\tilde ad\tilde u_p+\tilde bd\tilde v_q
%\] 
%for some locally defined functions $\tilde  a$ and $\tilde  b$. Thus,
%\[
%\omega_\Xi=\tilde adu_p+\tilde bdv_q
%\] 
%for some locally defined smooth functions $  a$ and $b$ near $\Xi$.
%
%Since
%\[
%\int_\Xi \omega_\Xi=\Xi \cdot 
%\Xi=0
%\]
and let  $ \digamma$ denote a fundamental domain in $\tilde M$. Then,
\begin{eqnarray*}
\int_{\tilde \Xi \cap \digamma }|d \tilde v_q| &=& \left |\int_{\tilde \Xi \cap \digamma} d \tilde v_q \right| \ (\mbox{because $d \tilde v_q$ is nonzero on $\Xi$})\\
&=& \left |\int_{ \Xi } d  v_q \right| \ (\mbox{because $d \tilde v_q$ descents to $d v_q$ in $M$})\\
&=&\left | \int_{M } \omega_\Xi \wedge d  v_q \right| \ (\mbox{because $\omega_\Xi$ is  Poincare dual of $\Xi$})\\
&\leq &c_\Xi(L^{-1}+1) \ (\mbox{by} \ (\ref{kappanorm})),
\end{eqnarray*}
where, since all $u_p$ are homotopic, $c_\Xi$ is a topological constant.
This implies that $d\tilde v_q$ are uniformly bounded in $L^1(\Xi \cap \digamma)$ and hence  $\tilde v_q$ are uniformly bounded in $L^\infty(\Xi \cap \digamma)$.
Since $\tilde v_q$ is invariant under the normalized gradient flow, it follows that $\tilde v_q$ is uniformly bounded on the open dense set of the fundamental domain $\digamma$ consisting of all non-critical trajectories. Hence, by continuity, $\tilde v_q$  is uniformly bounded on the closure of $\digamma$. Since the representation $\alpha_q$ is also uniformly bounded by (\ref{boundtrlength}), the local boundedness of $\tilde v_q$ in $\tilde M$ follows.
\end{proof}
%Corollary 1: Existence of representation 
%Corollary 2: $v_q \rightarrow v$ in $L^s$ for all $s$.

\section{The limit $q \rightarrow 1$}\label{qgoesto1}
In this section we construct a DeRham 1-current $V=dv$ obtained as a limit as $q \rightarrow 1$ of the closed forms $dv_q$ associated to the normalized conjugate harmonic functions to $u_p$. We  show that there exists a limiting representation $\alpha: \pi_1(M) \rightarrow \R$ and an $\alpha$-equivariant function $\tilde v$ whose derivative induces the current $V$. We further show that the function $\tilde v$ is locally in $L^\infty$ and locally of bounded variation. In Section~\ref{consmeasure} we will show that $\tilde v$ is  of least gradient (1-harmonic).

\begin{proposition} \label{lemma:limmeasures0} Given a sequence $q \rightarrow 1$, there exists a subsequence (denoted again by $q$) such that:
\begin{itemize}
\item $(i)$ There exists a closed  1-current $ V \in \mathcal D_1( M)$ such that 
$dv_q \rightharpoonup  V.$\\
\item $(ii)$ There exists a closed  1-current $ \tilde V \in \mathcal D_1( \tilde M)$ such that 
$d\tilde v_q \rightharpoonup  \tilde V.$ Furthermore, if $\sigma: \tilde M \rightarrow M$ denotes the universal covering map, then
$\sigma^*(V)= \tilde V.$\\
\item $(iii)$ There exists a representation 
$\alpha: \pi_1(M) \rightarrow \R$
such that for any $\gamma \in \pi_1(M)$,
$\alpha(\gamma)=\lim_{q \rightarrow 1} \alpha_q(\gamma).$
Furthermore,
$\alpha(\gamma)=   V( \omega_\gamma)$
 where $\omega_\gamma$ is the Poincare dual to the homology class defined by $\gamma$.
 \item $(iv)$ The homology class $[V] \in H_1(M, \R)$ is dual to $\alpha$.
 \end{itemize}
 \end{proposition}
\begin{proof} 
For $\phi \in  \Omega^1(M)$ a test function and $|\phi|_{L^\infty} \leq 1$, we have by (\ref{kappanorm}) that
\[
  \left | \int_M \phi \wedge dv_q  \right |  \leq C.
\]
  By weak compactness (cf. \cite[Lemma 2.15]{simon}), there exists  $V \in  \mathcal D_1( M)$ such that (after passing to a subsequence)
  \[
  dv_q \rightharpoonup  V 
  \]
  and $V$ is closed, being the  distributional limit of closed forms. 

For $(ii)$, the proof of the convergence is exactly the same as the proof of $(i)$. In order to prove the statement about the pullback, 
consider an open cover of $M$ given by basic sets and let $ \{ V_i \}$ be the cover of $\tilde M$ obtained by the preimage of the sets in $M$. Let $\zeta_i$ be a partition of unity subordinate to $ \{ V_i \}$. By definition, after identifying $V_i \simeq \sigma (V_i)$ and $dv_q= d\tilde v_q$,
\begin{eqnarray*}
\sigma^*(V)(\phi)&=&\sum_i  V(\zeta_i \phi)
= \lim_{q \rightarrow 1} \int_{\tilde M}  \sum_i \zeta_i \phi \wedge dv_q\\
&= & \lim_{q \rightarrow 1} \int_{\tilde M}   \phi \wedge d \tilde v_q
=  \tilde V(\phi).
\end{eqnarray*}

To prove $(iii)$, note that by the weak convergence of $dv_q$,
\[
\alpha_q(\gamma)= \int_\gamma  dv_q =\int_M  \omega_\gamma \wedge dv_q  \rightarrow  V( \omega_\gamma) =\alpha(\gamma).
\]

For  $(iv)$ notice that $\alpha$ factors through the abelianization of $\pi_1(M)$ to define an element in $H_1(M, \R)^*=H^1(M, \R)$ which is dual to $[V]$ by $(iii)$.
\end{proof}

\begin{definition}\label{def:bdvar} Let $U \subset M$ an open set and $f \in L^1(U)$. We define
\[
||df||_U=\sup\{ \int_M d\phi  \wedge f: \phi \in \mathcal D^1(U), \ \max|\phi| \leq 1 \} 
\]
 and set
\[
|f|_{BV(U)}= |f|_{L^1(U)}+||df||_U.
\]
We say that $f$ is of {\it{bounded variation}} in $U$ if $|f|_{BV(U)} < \infty$. 
\end{definition}

\begin{theorem} \label{lemma:limmeasures2} There exists a sequence  $q \rightarrow 1$  and $ \tilde v:  \tilde M \rightarrow  \R$ such that $\tilde v_q$ converges to $\tilde v$ {\it weakly} in $BV_{loc}(\tilde M)$ and {\it strongly} in $L^s_{loc}(\tilde M)$ for all $s \geq 1$.
Furthermore, $\tilde v$ has the following properties:
\begin{itemize}
\item $(i)$ $\tilde v$ is locally in $L^\infty$ and locally of bounded variation
%\item $(ii)$ $\tilde v$ is normalized so that
%  \[
%  \int_\digamma \tilde v *1=0.
%  \]
\item $(ii)$ $\tilde v$ is equivariant under $\alpha$, i.e for every $ \gamma \in \pi_1(\tilde M)$ and a.e. $ z \in \tilde M$
  \[
  \tilde v(\gamma z)=\tilde v(z)+ \alpha(\gamma)
  \]
%  \item $(iii)$ $\tilde v$ is locally a map of least gradient.
\end{itemize} 

\end{theorem}
\begin{proof} Fix $W \subset  \subset \tilde M$, and choose a finite number $\gamma_1,...,\gamma_N \in \pi_1(M)$ such that
\[
W \subset \subset \bigcup_{i=1}^N \gamma_i(\digamma).
\]
Since, by Proposition~\ref{lemma:limmeasures0}(ii),  $| \alpha_q(\gamma_i)| \leq C$  for $i=1,...,N $ and $j=1,2,...$, we obtain by the equivariance of $\tilde v_q$, (\ref{normintv3})
 and the fact that 
$\gamma_i$ act as isometries on $\tilde M$ that
\begin{equation}\label{estW}
\left| \frac{1}{vol(W)} \int_W \tilde v_q(z)dz \right| \leq  \frac{NC |\digamma|}{vol(W)} \leq C'.
\end{equation}
%Hence, we can choose a further subsequence so that
%\[
%\int_W \tilde v_j \rightarrow \bar v_0.
%\]
Now set,
\[
w_q(x)=\tilde v_q( x)- \frac{1}{vol(W)} \int_W \tilde v_q( z)dz; \ \ dw_q=d\tilde v_q.
\]
Similarly, by the Poincare inequality and (\ref{kappanorm0})
\[
|w_q|_{L^1(W)} \leq c |dw_q|_{L^1(W)}= c|d\tilde v_q|_{L^1} \leq C,
\]
which combined with (\ref{estW}), implies
\[
|\tilde v_q|_{W^{1,1}(W)} \leq C.
\]
Hence, there exists a subsequence (denoted again by $\tilde v_q$) and $\tilde v^W \in BV(W)$
such that
\[
\tilde v_q \xrightharpoonup{BV(W)} \tilde v^W.
\]
By a diagonalization argument we can define $\tilde v \in BV_{loc}(\tilde M)$ such that
\[
\tilde v_q \xrightharpoonup{BV_{loc}(\tilde M)} \tilde v.
\]
By the Rellich Lemma and the fact that $\tilde v_q$ are locally uniformly bounded by Proposition~\ref{locbd}, 
\[
\tilde v_q  \rightarrow \tilde v \in L^s_{loc} \ \ \forall s \geq 1.
\]
%Furthermore, by definition of $BV_{loc}(\tilde M)$,
%\[
%d\tilde v_q \rightharpoonup d \tilde v.
%\]
To show that  $\tilde v$ is locally bounded, fix $W \subset \tilde M$ compact. Again, since   
$|\tilde v_q|_{L^\infty} \leq C$  by Proposition~\ref{locbd}, and  $\tilde v_q  \rightarrow \tilde v$  in $L^s(W)$ for all $s$, it follows that $|\tilde v|_{L^s(W)} \leq C$ uniformly in $s$ and thus $\tilde v \in L^\infty(W)$.

Statement $(ii)$ follows from the equivariance 
$\tilde v_q(\gamma z)=\tilde v_q(z)+ \alpha_q(\gamma)$ 
and the fact that $L^s_{loc}$ convergence implies a.e convergence. 
 Since we have already shown that the functions $\tilde v_q$ converge strongly to $\tilde v$ in $L^s_{loc}$ for $s>1$. 
\end{proof}

\begin{remark} We will see in Section~\ref{consmeasure} that $\tilde v$ is a locally a function of least gradient. For Euclidean domains this follows also from  \cite{juutinen}, Proposition 4.5. We will give a proof of this fact in Theorem~\ref{thmlegr}.
\end{remark}

\begin{definition} Let $ L=\tilde M \times_\alpha \R$ be the flat affine bundle associated to the representation $\alpha$ and $v$  the section of  $ L$ induced from $\tilde v$. For an $L^1$-section $\xi: M \rightarrow L$, set $||d\xi||=||d\xi||_M$ and $|\xi|_{BV}=|\xi|_{BV(M)}$ as in Definition~\ref{def:bdvar}.  
 With this definition, $v$ becomes a  {\it section (twisted map) of bounded variation. } 
In view of Theorem~\ref{lemma:limmeasures2}, $d \tilde v=\tilde V$ and we  will denote
\[
V=dv.
\]
\end{definition}
\begin{remark}For the rest of the paper we fix sequential limits $u=\lim_{p \rightarrow \infty} u_p$ and $v=\lim_{q \rightarrow 1} v_q$ in the appropriate function spaces described above. We conjecture that $u$ and $v$ are essentially unique, though we are unable to prove this. See Conjectures~\ref{Conjecture 2} and~\ref{Conjecture 3}.
\end{remark} 

\begin{remark}\label{radonms}
Recall that, by the Riesz representation theorem \cite[Chapter 6, (2.14)]{simon},   given a $p$-current $S \in \mathcal D_p (U)$ of finite mass,
we can write
\[
S(\phi)=\int_{ U} \phi \wedge \vec {S} \ |dS|; \ \ \phi \in \mathcal D^{n-p} ( U) 
\]
for a Radon measure $|d S|$ and a measurable section $ \vec {S}$ of $\Lambda^p(\tilde M)$  where
$| \vec {S}|=1$ $|d S|$-a.e. It is customary to write the $p$-form valued Radon measure $\vec {S} \ |dS|$ by $S$ and use the notation
\[
S(\phi)=\int_{ U} \phi \wedge S.
\]
We will use this notation throughout the rest of the paper.
\end{remark}

\section{The geodesic lamination associated to the $\infty$-harmonic map}\label{sect:crandal}
For this section we allow $(M, g)$ to be a closed hyperbolic manifold of any dimension $n \geq 2$. We show that the gradient lines of the $\infty$-harmonic map $u$ at the points of maximum stretch define a geodesic lamination. The major difficulty lies in defining the gradient lines of $u$, because $grad(u)$ is not even known to be continuous.  We overcome this issue by adapting to the hyperbolic metric an argument due to Crandall for Euclidean space. This is a hyperbolic version of what is known as {\it{comparison with cones}} and which for Euclidean metrics is  equivalent to the notion of viscosity solutions of the $\infty$-Laplace equation. (\cite{crandal} or \cite{lindqvist}). We will not attempt to develop such a theory in this paper and we only prove the bare minimum that we need for our topological applications. For more details on open problems see Section~\ref{conjectures}.

\subsection{Statement of the theorem} We start by recalling the notion of a geodesic lamination.
%We give a proof modeled on a proof shown to us by Craig Evans for Euclidean space. We first clarify our use of the expression $|du(x)|$ for a Lipschitz map $u$. Here $L$ is the Lipschitz constant of the $\infty$ harmonic map $u.$
%
% Let  
% \[
% L_(B,x)  = (max y in B : |u(x) - u(y)|/dist(x,y)).
% \]
%\begin{definition}   $|dux)| = \lim_{ r \rightarrow 0} L_{B_(x)}$ for $B_r(x) = (y in M: distance (x,y) <=r}.$
%\end{definition}
%\begin{definition}  the set $\lambda$ is a geodesic lamination of $M$ if
%     1) $\lambda$ is closed
%     2) for $x \in \lambda$, $\lambda(x)$ is an embedded geodesic containing $x$. 
%     3) $\lambda(x)$ intersect $\lambda(y)$ is either $ \lambda(x) = \lambda(y)$ or empty.
%\end{definition}

\begin{definition}A geodesic lamination $\lambda$ is a closed subset of  $(M, g)$   which is a disjoint union of simple, complete geodesics. 
\end{definition}

The next theorem is the main result of the section. Recall from Section~\ref{lipconst} that $L_u(x)$ denotes the  local Lipschitz constant at $x$.

\begin{theorem}\label{straightline} Let  $(M, g)$ be a closed hyperbolic manifold of dimension $n \geq 2$   and let $ u: M \rightarrow S^1$ be $\infty$-harmonic (i.e a limit of $p$ harmonic maps for $p \rightarrow \infty$) with  Lipschitz constant $L:=|du|_{L^\infty(M)}$. Then,
\[
 \lambda_u=\{ x \in M: L_u(x)= L \} 
\]
is a geodesic lamination in $M$.
\end{theorem}

First, note the following straightforward:
\begin{lemma} \label{realized} Let  $(M, g)$ be a closed Riemannian manifold  and $ f: M \rightarrow S^1$ a  Lipschitz map  with global Lipschitz constant $L=L_f(M)$. Then, the set
\[
 \lambda_f=\{ x \in M: L_f(x)= L \} 
\]
is non-empty and closed.
\end{lemma}
\begin{proof} It follows  from Proposition~\ref{crandal0} $(i)$, on the upper semicontinuity of the local Lipschitz constant. Here are the details: By Proposition~\ref{crandal0}, take a sequence $x_i$ such that  $L_f(x_i) \nearrow L$. By compactness, we may assume $ x_i \rightarrow x$ and by upper semicontinuity $L_f(x) \geq \lim_i L_f(x_i) = L$. Thus, $x \in  \lambda_f$ and hence $ \lambda_f \neq \emptyset$. By  upper semicontinuity 
\[
 \lambda_f=\{ x \in M: L_f(x)= L \}=\{ x \in M: L_f(x) \geq L \} 
\]
is closed.
\end{proof}

%Theorem 5.3:  Let $\lambda = (x in M such that |du(x)| = L)$ . If the infinity harmonic map $u$ is the limit of a sequence of p harmonic maps for $p \rightarrow \infty$, then $\lambda$  is a geodesic lamination.

%Proposition 5.4 (your 5.3). The set $\lambda$ is non-empty and closed.
%
%(Put in your proof).
\subsection{Comparison with cones}
%\begin{definition}\label{copco} Let $M$ be a compact, hyperbolic manifold, i.e $\tilde M =H^n$ and denote by
%$d(.,.)$ the distance function.
%A cone function with vertex $x_0 \in H^n$ is a function of the form
%\[
%c(x)=A + B d(x,x_0),
%\]
% where $a,b \in \R.$ 
% Let $\Omega \subset H^n$ an open, connected subset (possibly $\Omega = H^n$). 
% A function $u \in C(\Omega)$ satisfies
%{\it{comparisons with cones from above}} if it possesses the following property: 
%For every open set $V$ with 
% $\bar V \subset  \Omega$ compact and every cone function $c$ with vertex $x_0 \in H^n \backslash V$,
%\[
%u  \leq c \  \mbox{in} \ \partial V \Longrightarrow u \leq c \ \mbox{in} \ V.
%\]
%We say that $u$ satisfies {\it{comparisons with cones from below}} if $-u$ satisfies  comparisons with cones from above. Finally, $u$ satisfies {\it{comparisons with cones }} if it satisfies comparisons with cones from above and below.
%\end{definition}
In this section we prove that our minimizers satisfy comparison with cones. For Euclidean metrics this is known to be equivalent to the notion of viscosity solution of the $\infty$-Laplace equation (cf. \cite{crandal}). In the present article we deal primarily with hyperbolic metrics and we expect every local result known for the Euclidean metric to also hold for our case as well. Below we will only prove the bare minimum necessary to prove our theorem on geodesic laminations, leaving most analytic aspects for a future project.

We first note that the
 map $d(x,x_0)$ can be approximated by  cone $ p$-harmonic functions $c_p(x) = f_p(d(x_0,x))$. In Euclidean space $\R^n$,
 \[ 
 f_p(t) = t^{\frac{p-n}{p-1}}=t^{1-\frac{n-1}{p-1}}
 \] 
 and in hyperbolic space $H^n$, by a function $f_p(t)$ satisfying 
 \[
 \frac{df_p(t)}{dt} = (1/sinh(t))^{\frac{n-1}{p-1}}. 
 \]
 
\begin{lemma}The function $f_p(t)$ is $p$-harmonic and $f_p(t) \rightarrow t$  uniformly on compact sets of $H^n$.  
\end{lemma}
\begin{proof} The metric on $H^n$ can be written in polar coordinates as
\[
g=dt^2+\sinh^2 t d\theta^2
\]
where $dt$ is hyperbolic length and $d \theta$ is the metric on $S^{n-1}$. From this, it follows immediately that $f_p(t)$ is $p$-harmonic.
To show the second statement, write
\[
f_p'(t)=h_p'(t)(t/sinh(t))^{(n-1)/(p-1)}\left (1-(n-1)/(p-1)\right )^{-1} \ \mbox{where} \ h_p(t)=t^{1-(n-1)/(p-1)}
\]
from which we obtain
\[
a_p h_p'(t) \leq f_p'(t) \leq h_p'(t) b_p 
\]
where $a_p$ and $b_p$ are constants converging to 1 as $p \rightarrow \infty$. Thus
\[
a_p t^{1-(n-1)/(p-1)} \leq \int_0^t f_p'(s)ds \leq b_p t^{1-(n-1)/(p-1)} 
\]
hence, by normalizing $f_p$ so that $f_p(0)=0$,
\[
a_p t^{1-(n-1)/(p-1)} \leq  f_p(t) \leq b_p t^{1-(n-1)/(p-1)} 
\]
from which the convergence follows.
\end{proof}

Since we can approximate both the $\infty$-harmonic function and the cone by $p$-harmonic maps, we get the proof of the following 
\begin{proposition}\label{comocones1}  If
\[ 
u(x) \leq A + B d(x,x_0) = c(x)
\]
for $ x \in \partial B_r(x_0)$  and at $x = x_0$, then 
\[
u(x) \leq c(x) \ \forall x \in B_r(x_0).
\] 
\end{proposition}

\begin{proof}  Both the function $u$ and the cone $c $ are  uniform limits in $C^0$ of $p$-harmonic functions $u_p$ and $c_p$ respectively. Hence, for $x \in \partial B_r(x_0)$ and also for $x = x_0$. 
\begin{eqnarray*}
 u_p(x) &<& \epsilon + u(x) \\
 &\leq& \epsilon + A + B d(x,x_0)\\
 &<& \epsilon(1 + Br) + A + B f_p(d(x_0,x)) \\
 &\leq& \epsilon(1 + Br)  +c_p(x)
  \end{eqnarray*}
 Here $\epsilon = \epsilon(p) \rightarrow 0$ as $p \rightarrow \infty$. However, both $u_p$ and the cone $c_p$ are $p$-harmonic functions. By the strong maximum principle for $p$-harmonic functions applied to the punctured disc $B_r^*(x_0)$, we obtain
 \begin{eqnarray*}
u(x) &\leq &u_p(x) + \epsilon \\
&\leq & c_p(x) + \epsilon(2 + Br)\\
&\leq & c(x) + 2\epsilon(1 + Br).
  \end{eqnarray*}
Since $\epsilon = \epsilon(p)\rightarrow 0$ as $p \rightarrow \infty$, this finishes the proof. 
\end{proof}

%\begin{proposition}\label{comocones1} Let $u_p \rightarrow u$ be a variational solution as in Definition~\ref{infinharm}. Then, 
%$u$ satisfies comparison with cones.
%\end{proposition}
%
%\begin{proof} Let $V  \subset H^n$ open with $\bar V$ compact, and $u$, $c$ as in Definition~\ref{copco}. Both $u$ and the cone $c $ are  local uniform limits  of $p$-harmonic functions $u_p$ and $c_p=A+Bf_p(d(x_0,.))$ respectively. Hence, for $x \in \partial V$ 
%\begin{eqnarray*}
% u_p(x) &<& \epsilon + u(x) \\
% &\leq& \epsilon + A + B d(x,x_0)\\
% &<& \epsilon(1 + B) + A + B f_p(d(x_0,x)) \\
% &=& \epsilon(1 + B)  +c_p(x)
%  \end{eqnarray*}
% Here $\epsilon = \epsilon(p) \rightarrow 0$ as $p \rightarrow \infty$. However, both $u_p$ and the cone $c_p$ are $p$-harmonic functions. By the strong maximum principle for $p$-harmonic functions, we obtain
% \begin{eqnarray*}
%u(x) &\leq &u_p(x) + \epsilon \\
%&\leq & c_p(x) + \epsilon(2 + B)\\
%&\leq & c(x) + \epsilon(3 + B)
%  \end{eqnarray*}
%on $V$. Since $\epsilon = \epsilon(p)\rightarrow 0$ as $p \rightarrow \infty$, this finishes the proof. 
%\end{proof}
%
%By applying Proposition~\ref{comocones1} for the punctured ball of radius $r$ we immediately obtain (see also \cite[formula (6.1)]{lindqvist})
%
\begin{corollary}\label{comocones2}
The ratio
\[
\max_{d(x,x_0)=r}\frac{u(x)-u(x_0)}{r}
\]
is increasing in $r$. The same holds with $u$ replaced by $-u$.
\end{corollary} 

\begin{proposition}\label{comocones2}  Let $x_0 \in \lambda_u$ be arbitrary and $B_r(x_0) \subset M$ (we can lift to the covering space if we choose). Assume that as $ x_i \rightarrow x_0$, 
$ \frac{u(x_i)-u(x_0)}{d(x_i,x_0)} \rightarrow +L \ (\mbox{resp.} -L).$  
Then 
\[
u(x)= u(x_0) + Lr \ (\mbox{resp.} -Lr)
\]
 for some point $x \in \partial B_r(x_0)$, and the geodesic between $x_0$ and $x$ lies in $\lambda_u.$
\end{proposition}

\begin{proof} Let $B = \max_{x \in \partial B_r} \frac{u(x)- u(x_0)}{r}.$  Let $x \in \partial B_r(x_0)$ on which $B$ is taken on.  Since $L$ is the  Lipschitz constant $B \leq L.$
Suppose $B < L$, then  by Corollary~\ref{comocones2} $u (x) \leq u(x_0) +Bd(x,x_0)$. But 
\[
\lim_i \frac{u(x_i) - u(x_0)}{d(x_i,x_0)} = L > B,
\] 
and for some $x_i$,
\[ 
u(x_i) - u(x_0) > Bd(x_i,x_0).
\] 
This gives a contradiction to the statement that $u$ lies under the cone $c$. So $B = L.$

Let $\lambda_0$ be the geodesic parameterized by arc length between $x_0$ and $x$.  Then,  since $L$ is the best Lipschitz constant, for $0 \leq s < t \leq r$
\[
 u(\lambda_0(t)) -u(\lambda_0(s)) \leq L(t-s).
 \]
But 
\[
u(x)- u(x_0) = u(\lambda_0(r) ) - u(\lambda_0(0)) = Lr. 
\]   
This gives estimates above and below on $u(\lambda_0(t))$ that shows  $u(\lambda_0(t)) = u(x_0) + Lt.$
For the case of $-L$, apply the same procedure to $-u$.
\end{proof}

\begin{straightline} We have shown that every point $x \in \lambda_u$ lies in a geodesic in $\lambda_u.$ We need only show that a) the entire geodesic lies in $\lambda_u$ and b) if the geodesics intersect, they form an angle of 0 or $\pi$.  We show b) first, as it is part of a). Suppose two geodesics $\lambda_1$ and $\lambda_2 \in \lambda_u$ meet at $x_0$, and that $x_0 = \lambda_1(0)$ is an interior point of $ \lambda_1.$ Assume also that $\lambda_2(0) = x_0$ and that the geodesics are parameterized by arc length.  Then
\[ 
u(\lambda_1(t)) = u(x_0) + Lt
\]
 for $t$ of both signs and 
 \[
 u(\lambda_2(s)) = u(x_0) + Ls
 \]
  for either $s > 0$ or $s < 0.$ Using the fact that $ L$ is the best Lipschitz constant, 
  \[ 
  Ld (\lambda_1 (t), \lambda_2(s)) \geq
 |u(\lambda_1(t)) - u(\lambda_2(s))| = |L(t - s)|. 
  \] 
It follows that  $|t-s|$ is the distance from $\lambda_1(t)$ to $\lambda_2(s)$ along the geodesics and must be greater than or equal to the actual distance. But we already know from the inequality that it is less than or equal to the distance between them on $M$. Equality follows; hence the geodesics meet at angle 0 if we parameterize them both in the direction of increasing $u.$

To show that the entire geodesic $\lambda_0$ lies in $\lambda_u$, we suppose not.  Then, choose $x_0$ near but not at the end the part of of geodesic $\lambda_0$ which lies in $\lambda_u.$  The Lipschitz constant at $x_0$ is taken on in both directions. More precisely, there are sequences 
$ x_i^\pm \rightarrow x_0$, 
$ \frac{u(x_i^\pm)-u(x_0)}{d(x_i^\pm,x_0)} \rightarrow \pm L .$ This follows by a straightforward argument using comparison with cones (cf. \cite[Lemma 4.6]{crandal}).
Hence by Proposition~\ref{comocones2}, there are two geodesic rays emanating from $x_0$ on which take on the best Lipschitz constant from above and below, both of  which are in $\lambda_u$ until they reach the boundary of $B_r(x_0).$  By the previous argument, these rays must make an angle of either 0 or $\pi $ with $\lambda_0.$ Hence $\lambda_0$ intersect $B_r(x_0)$ in $\lambda_u.$
\end{straightline}

\subsection{Another interpretation of the best Lipschitz constant} 
In this section we fix $(M,g)$ a closed hyperbolic manifold. Let $u: M \rightarrow S^1$ be an $\infty$-harmonic map in a given homotopy class  with best Lipchitz constant $L=L_u$. Let
$\tilde u : \tilde M \rightarrow \R$
denote the lift to the universal cover, equivariant  under the homomorphism $\rho:  \pi_1(M) \rightarrow \Z$. 
%i.e 
%\[
%\tilde u(\gamma z)=  \tilde u( z)+ \rho(\gamma), \ \forall \gamma \in \Gamma\  \mbox{and} \   \forall z \in \tilde M.
%\]
%For $\gamma \in \pi_1(M)$, let $l_g(\gamma)$ denote the length of the geodesic homotopic to $\gamma$. Equivalently, via the identification $\Gamma \simeq \pi_1(M)$,
%\[
%l_g(\gamma)=\inf_{\tilde x \in H^n}d_{H^n}(\tilde x, \gamma(\tilde x)).
%\]
Let $\mathcal S$ denote the set of free homotopy classes of simple closed curves in $M$.  Given $\gamma \in \mathcal S$,
let $l_g(\gamma)$ denote the length of the geodesic representative of $\gamma$ and
define the functional
\begin{equation}\label{normlength}
K: \mathcal S \rightarrow \R_{\geq 0}, \ \  K(\gamma)=  \frac{|\rho(\gamma)|}{l_g(\gamma)}.
\end{equation}
In the above, by a slight abuse of notation, we denote by $\gamma$ also the element in $\pi_1(M)$ corresponding to the free homotopy class $\gamma$.
Let 
\[
K= \sup_{\gamma \in  \mathcal S} K(\gamma)
\]
and note that 
\begin{equation}\label{ineq0}
K \leq L.
\end{equation} 
Indeed, given $\gamma \in \mathcal S$ denote by $\tilde \gamma: [0, T] \rightarrow \tilde M$ the lift of any loop in  $\gamma$ parametrized by arc length. Note that,
\begin{eqnarray}\label{water}
|\rho(\gamma)|&=&  |\tilde u(\tilde \gamma(0))-  \tilde u(\tilde \gamma(T))| 
 = \left |\int_0^T  \frac{d(\tilde u \circ \tilde \gamma)}{dt}  dt  \right |.
 \end{eqnarray}
By taking $\tilde \gamma$ a lift of the geodesic representative in the free homotopy class $\gamma$, and noting $T=l_g(\gamma)$,
\begin{eqnarray}\label{water2}
|\rho(\gamma)|
 &\leq& \int_0^T \left |d \tilde u_{ \tilde \gamma(t)} \right | dt  \leq LT. \nonumber
\end{eqnarray}
Hence
\[
\frac{|\rho(\gamma)|}{l_g(\gamma)} \leq L,
\]
which implies (\ref{ineq0}). 

The following theorem is a version of \cite[Theorem 8.5]{thurston}. 
  \begin{theorem}\label{K=L} $K=L$.
\end{theorem}
\begin{proof}  
Let $\beta$ be a leaf of the maximum stretch lamination $\lambda_u$  parameterized according to arc length. Because $M$ is compact, for any $n$, we can find $t_1 < t_2 - 1$ such that 
$d_g(\beta(t_2), \beta(t_1))< 1/n$. (If $\beta$ is closed this holds trivially for any $n$ by taking $\beta(t_2)=\beta(t_1))$. Choose the closed geodesic $\gamma_n$ to be the geodesic homotopic to the broken  geodesic $\beta_n$ made up by following $\beta$ from $t_1$ to $t_2$ and then connecting $\beta(t_2)$ to $\beta(t_1)$ by a short geodesic of length less than $1/n$.
Note that
 $l_g(\gamma_n) \leq l_g(\beta_n) < t_2 - t_1 + 1/n$.
By (\ref{water}) and noting that $\beta$ is a curve of stretch $L$ for $u$, 
\begin{eqnarray}\label{water2}
|\rho(\gamma_n)|=|\rho(\beta_n)| \geq L(t_2 - t_1 - 1/n).
\end{eqnarray}
Hence,
\[
K \geq K(\gamma_n) > \frac{L(t_2 - t_1 - 1/n)}{t_2 - t_1 + 1/n} \rightarrow L
\]
as $n \rightarrow \infty$.
\end{proof}

\section{The concentration of the Measure}\label{consmeasure}
In this section, we will use the Euler-Lagrange equations to determine properties of the limiting measures on the maps which take on the best Lipschitz constants. 
The statements are actually statements about $L^1$ norms being small, which implies that the limiting measures are zero away from the set of maximum stretch $\{ L_u=L \}=\lambda_u$.  They make sense in the limit applications only when a continuous function is inserted in the integrals. However, the limits are still zero, since the sup norm of a test function is bounded by the modulus of continuity. Note that in this section we will not make use of the results of Section~\ref{sect:crandal} that $\lambda_u$ is a geodesic lamination. Also the results about the concentration of the measure work in any dimensions and any Riemannian metric.

In Section~\ref{LGRad} we will specialize to the case $n=2$ and show that the map $v$ obtained as a limit of the maps $v_q$ as $q \rightarrow 1$ is a map of least gradient. We will not explore this property further in this paper, however in Section~\ref{conjectures} we will indicate how this property can be used together with results about minimizing currents to give another proof that $\lambda_u$ is a geodesic lamination on the support of the measure $dv$. 

Finally in Section~\ref{sectevp} we will show how the results of the previous sections can be generalized to cover the equivariant problem for a general real valued homomorphism $\rho$.  There are no real changes. Our paper could have been written to include this more general situation from the start.  We did not do this, as many in our target audience would have found it a source of added confusion in a paper that already contains unfamiliar topics.

\subsection{The support of $V=dv$} Let  $(M,g)$ be a Riemannian manifold  of dimension $n \geq 2$ and 
let $u: M \rightarrow S^1$ be an $\infty$-harmonic map obtained as a sequential limit $u=\lim_{p \rightarrow \infty}u_p$ as in Theorem~\ref{nonsingular}. In order to simplify the notation, for this section only, we  renormalize the measure on $M$ so that the best Lipschitz constant $L=L_u= 1$.  Carrying factors of this constant around makes everything more difficult to write and read.

 As with the case of dimension 2, we continue  with the normalization $k_p$ as in  (\ref{normintv1}) 
 and by Lemma~\ref{klemma1}, 
\begin{equation}\label{klemma11}
\lim_{p \rightarrow \infty} k_p = 1.
\end{equation}
We define the 1-form $ U_p = k_p d  u_p$ and the closed $n-1$ form $V_q = |  U_p|^{p-2}*  U_p$. As in  (\ref{kappanorm0}) and (\ref{kappavolform}),
\begin{equation} \label{kappavolform21}
\int_M |V_q|*1=\int_M |  U_p|^{p-1}*1
 \approx 1 \ \ \mbox{(for $p$ large)}
\end{equation}
and
\begin{equation} \label{kappavolform22}
\int_M d  u_p \wedge V_q=1.
\end{equation}
As in Proposition~\ref{lemma:limmeasures0},
for $\phi \in  \Omega^1(M)$ a test function and $|\phi|_{L^\infty} \leq 1$,  (\ref{kappavolform21}) implies,
  $\left | \int_M \phi \wedge V_q  \right | $ is uniformy bounded,  
  hence 
  \[
  V_q \rightharpoonup  V, 
  \]
  where $V$ is a closed, $n-1$ current. In our previous notation, for $n=2$, $V=dv$.

The main result of this section is the following theorem:
\begin{theorem}\label{thm:supptmeasure} The support of the current $V$  is contained in the locus of maximum stretch $ \lambda_u$ of $u$.
\end{theorem}

\begin{lemma}  \label{klemma2}  Suppose $0 \leq e_p \leq e \leq 1$. Then 
\[
 e_p^{p-2}(e^2 - e_p^2) < 2/(p-2).
 \]
\end{lemma}

\begin{proof}  Let  $s_p = \frac{e_p}{e}$.    Then the expression we are trying to bound can be written as
\[
 e^p {s_p}^{p-2}(1 - s_p^2).
 \]
   But by calculus, the maximum of $s_p^{p-2}(1-s_p^2)$ is less than $2/(p-2)$. Since $e \leq 1$, we are done.
\end{proof}

\begin{lemma}  \label{kprop 3} Let  $U_p=k_pdu_p$, $U=du$ and
\[ 
G(p) =  2<U_p ,U_p -U> =|U_p|^2+ |U_p -U|^2 - |U|^2.
\]
  Define $Y_p$ as the set on which $G(p) \geq 0$.    Then
 \[   
 \lim_{p\rightarrow \infty}  \int_{Y_p}  |U_p|^{p-2}G(p)*1 = 0.
 \]
\end{lemma}
\begin{proof}   The difference $u_p - u$ is a function on $M$.  Hence, from the Euler-Lagrange equations for $u_p$ we have
\[
 \int_M  |du_p|^{p-2}<du_p,du_p - du>*1 = 0.
\]
Multiply by $k_p^p$ and substitute the expressions for $U_p$ and $U$ to get
\begin{equation}\label{kpU}
\int_M |U_p|^{p-2}<U_p, U_p - k_pU>*1 = 0.
\end{equation}
By (\ref{klemma11}) and (\ref{kappavolform21}),
\begin{eqnarray*}
\lefteqn{ \lim_{p\rightarrow \infty}  \int_M |U_p|^{p-2}(<U_p, U_p - k_pU>-1/2G(p))*1} \nonumber\\
&=&  \lim_{p\rightarrow \infty}  \int_M |U_p|^{p-2}<U_p, (U - k_pU)>*1 \nonumber\\
&\leq& \lim_{p\rightarrow \infty} (1-k_p) \int_M |U_p|^{p-1}|U|*1\\
&=& 0 \nonumber.
\end{eqnarray*}  
Combining with (\ref{kpU}),
\begin{equation}\label{kpU234}
 \lim_{p\rightarrow \infty} \int_M |U_p|^{p-2}G(p)*1 = 0
\end{equation}
and our proposition is proved if we can show that
\[
\lim_{p\rightarrow \infty} \int_{M \backslash Y_p} |U_p|^{p-2}G(p)*1 = 0.
\] 
Therefore, we need to bound the integral of
\[
 |U_p|^{p-2}(|U|^2 - |U_p|^2 - |U_p - U|^2)
 \]
  over the set where it is positive.
But this expression is bounded by 
 \[
 |U_p|^{p-2}(|U|^2 - |U_p|^2) < 2/p-2
 \]
on the larger set where $|U| \geq |U_p|$ by applying Lemma~\ref{klemma2} (for $e_p=|U_p|$ and $e=|U|$). This gives the desired bound.
\end{proof}

\begin{proposition} \label{kprop 4}  
\[
\lim_{p\rightarrow \infty}\int_M |U_p|^{p-2}|U_p - U|^2 *1 = 0.
\]
\end{proposition}

\begin{proof}   We have from Lemma~\ref{kprop 3} that
\[
\lim_{p\rightarrow \infty} \int_{Y_p} |U_p|^{p-2}(|U_p|^2 +|U_p -U|^2-|U|^2)*1  = 0.
\]

On the set ${|U|^2 \leq |U_p|^2 + |U_p -U|^2}$, this gives the desired estimate of the integral over that set.
On the compliment ${|U|^2 > |U_p|^2 + |U_p-U|^2}$  from Lemma~\ref{klemma2} (for $e_p=|U|_p$,  $e=|U|$)  we have that point-wise
\[
|U_p|^{p-2}|U_p - U|^2 < |U_p|^{p-2}(|U|^2 -|U_p|^2) < 2/(p-2).
\]
This bounds the integral on the entire manifold.
\end{proof}

%\begin{corollary} \label{uniquedu} For any two choices for $U = du$ and $U'=du'$,  we have
%\[ 
% \lim_{p\rightarrow \infty} \int_M |U_p|^{p-2}|U-U'|^2*1 = 0.
% \] 
%Hence  the derivatives of any two functions realizing the best Lipschitz constant must vanish on the support of $T$. 
%\end{corollary}
% 

%\item $(c)$ We can also show that 
%\[
%\lim_{p\rightarrow \infty} \int_M |<dv_p,*U>|*1 = 0
%\]
%since $<U_p,*U_p> = 0. $ This partly shows  that the measure $dv$ is supported along the fiber.

\begin{proposition} \label{prop:supptmeasure0}   If $\phi$ has support on the set where $|U| \leq \lambda < 1$, then
\[
\lim_{p\rightarrow \infty} \int_M |U_p|^p| \phi| *1 = 0.
 \]
\end{proposition}
\begin{proof} 
We go again to the estimate from Lemma~\ref{kprop 3} 
\[
\int_{Y_p} |U_p|^{p-2}\left((1-\lambda)|U_p|^2 +( \lambda|U_p|^2 +|U_p-U|^2 - |U|^2) \right )*1 \rightarrow 0.
\]
This provides a bound for the 
 integral of $(1 -\lambda)|U_p|^p$ over the set $|U|^2 \leq \lambda|U_p|^2 + |U-U_p|^2$.   In general, over the complimentary set, we do not have a bound.  However, if we are integrating over a set where  $|U| \leq \lambda$, just using the inequality that on that set 
 \[
 |U_p| \leq  \lambda^{-1/2} |U|
 \] 
 we have the pointwise bound
\[ 
|U_p|^p \leq |U|^p \lambda^{-p/2} \leq \lambda^{p/2}.
\]
Since $\lambda < 1$, the point-wise limit is 0. This bounds the integral over the entire manifold.
\end{proof}

\begin{thm:supptmeasure} 
As in Remark~\ref{radonms}, let $|V| $ denote the Radon measure associated to the distribution $V$.
By \cite[Chapter 6, (2.14)]{simon}, the weak convergence $V_q  \rightharpoonup V$
implies that for any open set $W \subset M \backslash  \lambda_u$
\begin{eqnarray*}
|V|(W) &\leq& \liminf_{q \rightarrow 1} | V_q|(W) \\
&=&\liminf_{p \rightarrow \infty} \int_W |U_{p}|^{p-1}*1\\
&\leq& \liminf_{p \rightarrow \infty} \left( \int_W |U_p|^p*1\right)^{\frac{p-1}{p}}\\
&=& 0
\end{eqnarray*}
The last equality is from Proposition~\ref{prop:supptmeasure0}.
\end{thm:supptmeasure} 

\begin{corollary} \label{currentsupportfiber} There is a sequence $p \rightarrow \infty$ (or equivalently $q \rightarrow 1$) such that 
\[
\lim_{q\rightarrow 1} \int_M |*du \wedge dv_q|*1 = 0.
\]
%Furthermore, if $n=2$,
%\[
%*du \wedge dv=0.
%\]
\end{corollary}
\begin{proof}
We have
\begin{eqnarray*}
\lefteqn{\lim_{q\rightarrow 1} \int_M |*du \wedge dv_q|*1}\\
&=&\lim_{p\rightarrow \infty}\int_M |U_p|^{p-2} |du \wedge U_p| *1\\
&\leq& \lim_{p\rightarrow \infty}\int_M |U_p|^{p-2}| U_p \wedge U_p| 
+\lim_{p\rightarrow \infty}\int_M |U_p|^{p-2}|<du-U_p, U_p>| *1\\
&=& \lim_{p\rightarrow \infty}\int_M |U_p|^{p-2}|<du-U_p, U_p>| *1\\
&\leq&\lim_{p\rightarrow \infty} \int_M |U_p|^{p-1}|du-U_p| *1\\
&\leq&\lim_{p\rightarrow \infty} \int_M |U_p|^{(p-2)/2}|du-U_p||U_p|^{p/2} *1\\
&\leq&\lim_{p\rightarrow \infty} \left(\int_M |U_p|^{p-2}|du-U_p|^2*1\right)^{1/2} \left(\int_M |U_p|^p *1\right)^{1/2}\\
&=& 0 \ \ \mbox{(by Proposition~\ref{kprop 4}, (\ref{klemma11}) and  (\ref{normintv1}))}.
\end{eqnarray*}
\end{proof}

\subsection{Stronger version of the support argument} In this section we show that  Proposition~\ref{kprop 4} and Proposition~\ref{prop:supptmeasure0} can be modified to cover the case where we replace $U=du$ by the derivative of any Lipschitz map $u'$ in the same homotopy class of $u$.
More precisely, let
\[
u': \tilde M \rightarrow \R
\]
be a $\rho$-equivariant  Lipschitz map and let 
\[
c'=\max |du'|, \  \ U'=\frac{1}{c'} du'.
\]
We continue with the normalization of the best Lipschitz constant $L = 1$ and since $u'$ is in the same homotopy class of $u$ we have $c' \geq 1$.
\begin{proposition} \label{kprop 4*}  
\[
\lim_{p\rightarrow \infty}\int_M |U_p|^{p-2}|U_p - U'|^2 *1 \leq C (c'-1).
\]
\end{proposition}
\begin{proof} We have to adapt the proof of Lemma~\ref{kprop 3} and Proposition~\ref{kprop 4}. First, we set
\[ 
G(p) =  2<U_p ,U_p -U'> =|U_p|^2+ |U_p -U'|^2 - |U'|^2.
\]
Equation (\ref{kpU}) has to be modified to
\begin{equation*}
\int_M |U_p|^{p-2}<U_p, U_p - k_pc'U'>*1 = 0,
\end{equation*}
hence
\begin{eqnarray*}\label{kpUc}
\int_M |U_p|^{p-2}<U_p, U_p - k_pU'>*1 &\leq& \int_M |U_p|^{p-1}| k_pU' - k_pc'U'|*1 \\
&\leq& Ck_p (c'-1)
\end{eqnarray*}
and equation (\ref{kpU234}) to
\begin{eqnarray}\label{kpU234c}
\lim_{p\rightarrow \infty} \int_M |U_p|^{p-2}G(p)*1 \leq C (c'-1).
\end{eqnarray} 
This  error persists through the rest of the proof without any additional changes, from which the result follows.
\end{proof}

A  consequence is the following generalization of Theorem~\ref{thm:supptmeasure} about the stretch locus of any best Lipchitz map:
\begin{corollary}\label{strsuppp}
 The support of the current $V$  is contained in the locus of maximum stretch $ \lambda_{u'}$ for any best Lipschitz map $u'$.
\end{corollary}
\begin{proof}With the normalization $L=1$ the best Lipschitz map $u'$ has $c'=1$. We continue the proof of Theorem~\ref{thm:supptmeasure} by using 
Proposition~\ref{kprop 4*} instead of Proposition~\ref{kprop 4}. Since $c'=1$ both Propositions yield the same answer, so there is no difference in the argument.
 \end{proof}
 
 We can rephrase the corollary above as follows: 
  Following \cite[Definition 1.2]{kassel}, let $\mathcal F$ denote the collection of $\rho$-equivariant  best Lipschitz functions $u': \tilde M \rightarrow \R$ and define 
  \[
\lambda=\cap_{u' \in \mathcal F} \lambda_{u'}.
\]
By \cite[Lemma 5.2]{kassel}, $\lambda$ is a geodesic lamination which
 plays the role of Thurston's chain recurrent lamination (cf. \cite[Theorem 8.2]{thurston}).
 Corollary~\ref{strsuppp} can be restated by saying that the support of the current $V$  is contained in $\lambda$.

\subsection{Maps of least gradient}\label{LGRad} In this section we assume $n=2$ and write $V=dv$ where $v: M \rightarrow L$ is the section corresponding to $\tilde v: \tilde M \rightarrow \R$ equivariant under $\alpha$.
Least gradient is usually defined with respect to the Dirichlet problem in a domain.  When we have a section  $v$  we can, of course, define it with respect to the Dirichlet problem on domains in M, but we will give a more global definition. We will first prove the following consequences of Proposition~\ref{kprop 4}.

\begin{corollary}\label{withoutcont}  If $n = 2$ then,
\[
 \ \lim_{q\rightarrow 1} \int_M du \wedge dv_q = 1 \ \mbox{and} \ \lim_{p \rightarrow \infty} \int_M du_p \wedge dv =1.
\]
\end{corollary}
\begin{proof}  
The proof of the first equality is similar to Corollary~\ref{currentsupportfiber}:
\begin{eqnarray*}
\lefteqn{\lim_{q\rightarrow 1} \int_M du \wedge dv_q}\\
&=&\lim_{p\rightarrow \infty} \left( \int_M (du-du_p) \wedge dv_q+ \int_M du_p \wedge dv_q \right)\\
&=&\lim_{p\rightarrow \infty} \int_M (du-du_p) \wedge dv_q  +1 \ \ \mbox{(by (\ref{kappavolform22}))} \\
&=& 1 \ (\mbox{as in the proof of Corollary~\ref{currentsupportfiber}}).
\end{eqnarray*}

For the second, 
continue the normalization of $L = 1$.  We use Proposition~\ref{maxest}
\begin{equation}\label{normcp}
c_l := max|du_l| \rightarrow L=1 \ \mbox{as} \  l \rightarrow \infty.
\end{equation}
Let 
\[
{U_l}'=  \frac{1}{c_l}du_l. 
\]
By Proposition~\ref{kprop 4*},
\begin{equation}\label{variant64}
 \lim_ {p \rightarrow \infty} \int_M |U_p|^{p-2} |U_p - U'_l|^2 *1 \leq C(c_l-1).
\end{equation}
Note,
\begin{eqnarray}\label{cuo13}
\lefteqn{\left |\int_M (U_p - U'_l)  \wedge dv_q  \right |} \nonumber \\
&= & \left |\int_M |U_p|^{p-2} <U_p - U'_l, U_p> *1  \right | \nonumber \\
&\leq & \int_M |U_p|^{p-1}|U_p-U'_l| *1   \\
 &\leq &\left( \int_M |U_p|^p *1 \right)^{1/2}\left( \int_M |U_p|^{p-2} |U_p - U'_l|^2 *1 \right)^{1/2}   \nonumber\\
 &\leq & {k_p} ^{1/2}C (c_l-1)^{1/2}  \ (\mbox{by (\ref{variant64}) and (\ref{normintv1})}). \nonumber
\end{eqnarray}
Thus, (\ref{klemma1}) implies
\begin{equation}\label{cuo1}
 \lim_ {p \rightarrow \infty}  \left|\int_M (du_p - U'_l)  \wedge dv_q  \right | \leq C (c_l-1)^{1/2}.
\end{equation}
 By (\ref{kappanorm}),
\begin{equation}\label{cuo2}
\left| \int_M (U'_l - du_l )  \wedge dv_q \right | \leq \max |U'_l - du_l| \leq c_l-1.
\end{equation}
By combining (\ref{cuo1}) and (\ref{cuo2})
\begin{equation}\label{cuo3}
\left| \int_M (du_p - du_l )  \wedge dv_q \right |  \leq  c_l-1+C (c_l-1)^{1/2}.
\end{equation}
Take now $p \rightarrow \infty$ ($q \rightarrow 1$) and use $dv_q  \rightharpoonup dv$,
 and (\ref{kappavolform}) to obtain
 \begin{equation}\label{cuo4}
\left| 1 -\int_M du_l   \wedge dv \right |  \leq c_l-1+C (c_l-1)^{1/2}.
\end{equation}
By (\ref{normcp}) $c_l \rightarrow 1$, hence  
the result follows.
\end{proof}

\begin{theorem}\label{thmlegr} The section $v$ is a section of least gradient in the sense that for  all functions $\phi$ on $M$ of bounded variation
\[
||dv|| \leq ||d (v+\phi)||. 
\]
\end{theorem}

\begin{proof} First of all, we show that $||d(u + \phi)|| \geq 1/L$.
  We pick a sequence $p$ such that $v_q \rightarrow v$.  Let $U'_p = \frac{1}{c_p}du_p$, where $c_p$ as in (\ref{normcp})
is a normalizing factor which sets  $\max |U'_p|=1$.  Then
\begin{eqnarray*} 
\lim_{p \rightarrow \infty} \int_M U'_p \wedge (dv + d\phi)   &=& \lim_{p \rightarrow \infty} \int_M U'_p \wedge dv \\
&=& \lim_{p \rightarrow \infty} 1/{c_p} \int_M    du_p \wedge dv\\
&=& 1/L  \ (\mbox{by   Corollary~\ref{withoutcont} and (\ref{normcp})}).
\end{eqnarray*}
We are going to complete the proof by showing that $||dv|| = 1/L$. Indeed, for any $\Phi \in\Omega^1(M)$ with $\max | \Phi| \leq 1$,
\begin{eqnarray*}
\int_M   \Phi \wedge dv &=& 
\lim_{p \rightarrow \infty}    \int_M   \Phi \wedge dv_q  \\
    &\leq&   \lim_{p \rightarrow \infty}  \int_M |dv_q|*1\\
    &\leq& 1/L \ (\mbox{by} \  (\ref{kappanorm0})).
\end{eqnarray*}
\end{proof}

\subsection{The equivariant problem}\label{sectevp} The results of the previous sections generalize in a straightforward way if we replace the map $f: M \rightarrow S^1$ by a 
$\rho$-equivariant map. 
More precisely, let $\rho \in H^1(M, \R)$. We can view $\rho$ as a homomorphism $\rho: \pi_1(M) \rightarrow \R$ and consider maps 
\[
\tilde f: \tilde M \rightarrow \R
\] 
satisfying the equivariance relation
\begin{equation}\label{equivfr}
\tilde f(\gamma \tilde x)=  \tilde f( \tilde x)+ \rho(\gamma), \ \forall \gamma \in \pi_1(M)\  \mbox{and} \   \forall \tilde x \in \tilde M. 
\end{equation}
In the case when $\rho$ is integer valued the map $\tilde f$ descends to a map $f:M \rightarrow S^1$ with induced homomorphism $f_*=\rho$ on the fundamental groups as studied in the previous sections. We will denote by $f$ the induced section of the flat affine bundle
$\tilde M \times_\rho \R$.

Next note that because of (\ref{equivfr}), the 1-form $d\tilde f$ is invariant under $\rho$. Hence it descends to a closed 1-form on $M$ which we denote by $df$.  We can proceed as before with minimizing  integral (\ref{pharmfun}) to obtain a $\rho$-equivariant map $\tilde u_p: \tilde M \rightarrow \R$ satisfying the $p$-harmonic map equation (\ref{pharm}). Furthermore, by taking $p \rightarrow \infty$ we obtain an infinity harmonic $\rho$-equivariant map $\tilde u: \tilde M \rightarrow \R$.  The map $\tilde u$ is a best Lipschitz map in the sense that it minimizes the Lipschitz constant among all $\rho$-equivariant maps. Theorem~\ref{thm:limminimizer} generalizes to this case.

The definition of the dual harmonic function $\tilde v_q$ in Section~\ref{sect:conjug} goes unchanged since its definition is purely in terms of $du_p$. The same goes with the convergence results as $q \rightarrow 1$ in Section~\ref{qgoesto1}.

The definition of the maximum stretch set $\lambda_u$ and proof of Lemma~\ref{straightline} only involves $L=|du|_{L^\infty}$ and thus makes sense for any equivariant map $\tilde u$. The theory on comparison with cones is local and thus it is not affected by going to equivariant maps.  The proof of Theorem~\ref{realized} remains unchanged. Finally, the results of this section on the support of the measure $V=dv$ and the least gradient property only involve the equivariant map $v$ and are not dependent on where $du$ came from. Thus there are no changes here as well. We state this in the form of the following theorem:

\begin{theorem}\label{thm:equivsit} Fix a homomorphism $\rho: \pi_1(M) \rightarrow \R$. There exists a $\rho$-equivariant infinity harmonic function $\tilde u: \tilde M \rightarrow \R$ and a least gradient function $\tilde v: \tilde M \rightarrow \R$ equivariant under a representation $\alpha: \pi_1(M) \rightarrow \R$. Furthermore, the support of the measure $dv$ is in the maximum stretch lamination defined by $\tilde u$.
\end{theorem}

\section{Construction of the transverse measure from the least gradient map}\label{sect7} 
In this section we assume $M$ is a closed hyperbolic surface, i.e $\tilde M=H^2$. We first review the concepts we need from topology to get the result about transverse measures.  These include Definition~\ref{flowbox}  flow boxes, Definition~\ref{orientflow} orientation, Definition~\ref{MDefinition5A}  transversals and Definition~\ref{deftransversecocycle} transverse cocycle. Following Bonahon, we connect the notion of  functions $\tilde v$ which are $\pi_1(M)$-equivariant and locally constant on $\tilde M \backslash \tilde \lambda$ with transverse cocycles.  In Theorem~\ref{transmeasure}, we use his theorem that a transverse cycle is a transverse measure if and only if it is non-negative to show that the least gradient map $v$ constructed in Theorem~\ref{lemma:limmeasures2} defines a transverse measure on the maximum stretch lamination $\lambda_u$ associated with the $\infty$-harmonic map $u$.   The definition of a transverse measure is equivalent to a function   on the universal cover with the right properties fits in well with our function of bounded variation $v$ (or $\tilde v$) which is constant on the components of $ M \backslash  \lambda$. See Theorem~\ref{transcoismeasthm}.

  \subsection{Flow boxes}
  
We start with the following elementary lemma from hyperbolic geometry
\begin{lemma}\label {lemma:flow} Let $\lambda$ be a lamination, and $f$ a geodesic orthogonal to a leaf $\lambda_0$. For $k \in \lambda \cap Im(f)$ and $\lambda_k$ be the leaf of $\lambda$ through $k$, let $n(k) =e^{ i \kappa(k)}$ be the unit normal direction of $\lambda_k$ when it intersects the  geodesic $f$ at $k$ and the same for $k'$. Then, there is a constant $c>0$ such that $|\kappa(k) - \kappa(k')| \leq cd_{H^2}(k,k')$.
\end{lemma}
\begin{proof} We use the unit disk model of hyperbolic space, and place the geodesic formed by $f $ on the $x$ axis, and the point $k$ at the origin. In other words, write 
\[
f: (-1,1) \rightarrow H^2 \simeq D^2; \ f(t)=t, \ k=0
\]
and the geodesic $\lambda_k$ is the straight line
\[
\lambda_k = \{n(k)t: \ -1 < t < 1\}.
\]
Similarly, the geodesic $\lambda_{k'}$ through $k'=f(k')>0$ (if $k'<0$ reverse the role of $k$ and $k'$) is the geodesic
\[
\lambda_{k'}= \{n(k')\frac{s +w}{1 + \bar w s}: \   -1 < s < 1\},
\]
where $w =\overline{ n(k')} k'$.  The geodesics $n(k) t$ and $n(k') s$ intersect, and $\lambda_k$ and $\lambda_{k'}$ do not, so for some $w'' = \overline{ n(k')}k''$, $0 < k'' \leq k'$, the geodesics $n(k) t$ and $n(k')\frac{s + w''}{1 + \bar w'' s}$ intersect at the endpoints on the unit circle $t=s=1$ or $t = s = -1$. Then
\[
\pm n(k) = \pm n(k')\frac{1 \mp w''}{1 \pm \bar w''}
\]
or equivalently,
\[
\frac{n(k)}{n(k')} = 1 \mp \frac{w'' + \bar w'' }{1 \pm \bar w''}.
\]
Then $| 1 - e^{i(\kappa(k) - \kappa(k'))}| \leq 2 \frac{|w''|}{1 - |w''|} \leq  \frac{2k'}{1-k'}$.
Since $d_{H^2}(k,k')= \tanh^{-1}(k') $, this inequality converts to the inequality in the Lemma provided $d_{H^2}(k,k')$ is not too large.
\end{proof}

\begin{definition}\label{flowbox} By a {\it flow box} or a {\it chart} for a geodesic lamination $\lambda$ we mean a bi-Lipschitz homeomorphism 
\begin{equation}\label{eqn:flowbox}
F: R^\sharp=[a,b] \times [c,d]  \rightarrow  F(R^\sharp) =R\subset M; \ F=F(t,s)
\end{equation}
such that there exists a closed set $K \subset (c,d)$ of Hausdorff dimension 0 such that
\[
F^{-1}(\lambda)= [a,b] \times K.
\]
\end{definition}
%{\bf{Change 4: These are classical flow boxes. Should I include a proof they exist and are Lipschitz????}}

 \begin{proposition} \label{prop:flow} Any geodesic lamination  on a closed hyperbolic surface (M,g) has an open neighborhood covered by a finite number of flow boxes (\ref{eqn:flowbox}). Furthermore,
   $F$  can be chosen so that  $\frac{\partial F}{\partial t}$  is a Lipschitz vector field along $F$.  \end{proposition}
\begin{proof} Let $f: [c,d] \rightarrow M$ be a Lipschitz transversal. We can assume as in Lemma~\ref{lemma:flow} that $f$ is a geodesic and let
\begin{equation}\label{normal}
n: [c,d] \rightarrow \R
\end{equation}
 denote the Lipschitz function defined as follows. Let $K = \{ k \in [c,d]: f(k) \in \lambda_u \}$ and $n(k) \in T_{f(k)}(M)$ be the unit tangent vector to the leaf of $\lambda$ through $k$, $k \in K$. By Lemma~\ref{lemma:flow},  $|n(k)- n(k')| \leq cd_{H^2}(k,k')$.  Extend $n$ to a Lipschitz function  on the interval. 
Define 
\[
F(s,t) = exp_{f(s)}(t n(s)). 
\]
Then
\[
\frac{\partial F}{\partial t}=dexp_{f(s)}(t n(s))n(s),
\]
is Lipschitz, and hence $\frac{\partial F}{\partial t}$   is Lipschitz. 
 \end{proof}
 
 \begin{remark}\label{nonsmooth}
 Note that $\frac{\partial F}{\partial s}$ is in $L^\infty$ but not necessarily continuous unless $\frac{d n} {ds} $ is. At the moment we are unable to 
 obtain such regularity for $\frac{d n} {ds} $. 
 \end{remark}
 Note that by construction,
 \begin{equation}\label{norone} 
 \left  |\frac{\partial F}{\partial t} (s,t) \right |=1 \ \ \mbox{for} \ \ s \in K.
 \end{equation}

 \begin{definition}\label{orientflow} A geodesic lamination $\lambda$ is called {\it orientable}, if there exists a Lipschitz unit vector field $n$ defined in a neighborhood of $\lambda$ and transverse to the leaves. 
  \end{definition}

 Note that by Lemma~\ref{lemma:flow}  a normal vector field exists locally, so the issue is existence of a global vector field. Also note that together with a choice of an ambient orientation for $M$, a choice of $n$ determines an orientation of the leaves. More precisely, the direction of the leaves followed $n$ must coincide with the orientation of $M$.
 
 Definition~\ref{orientflow} is clearly equivalent to any of the following conditions:\\
  $(i)$ There is a cover of a neighborhood of $\lambda$ with flow boxes $F$ as in Definition~\ref{flowbox} such that $F$ are orientation preserving
 with respect to  the ambient orientation of the manifold $M$ and the product orientation on $R^\sharp$. \\
 $(ii)$ Given $x \in \lambda$ and $\beta: (-\infty, \infty) \rightarrow M$ an orientation preserving parametrization of the leaf through $x$ with $\beta(0)=x$, there exists $\epsilon >0$ such that the map
 \begin{equation}\label{p1}
p_1\circ F^{-1}\circ  \beta: [-\epsilon, \epsilon] \rightarrow [a,b]
\end{equation}
 is orientation preserving, where $p_1$ denotes projection onto $[a,b]$.
 A cover of a neighborhood of $\lambda$ by flowboxes as above is called an {\it{oriented atlas}} of the lamination. An oriented atlas determines completely the orientation of $\lambda$.

 Throughout the section we fix an oriented atlas for $\lambda$ consisting of flow boxes $\{F\}$.

\begin{definition}\label{MDefinition5A}
For a continuous path $f: [l, m] \rightarrow M$, we let
\begin{equation}\label{defK}
K=f^{-1}\left(f([l, m]) \cap \lambda \right)
\end{equation}
and call $f$  {\it{transverse to the lamination}} $\lambda$ if for every $k \in K$ there exists a flow box 
$F: R^\sharp=[a,b] \times [c,d]  \rightarrow  F(R^\sharp) =R\subset M $
at $f(k)$ and $\eta=\eta(k)>0$ such that 
\begin{equation}\label{p2}
 p_2 \circ F^{-1} \circ f: [k-\eta, k+\eta] \rightarrow [c,d]
\end{equation}
is a homeomorphism onto its image, where $p_2$ denotes projection onto $[c,d]$.
We call $f$  an {\it{admissible transversal}}, if in addition $f(l), f(m)\in M_0=M \backslash \lambda$.
\end{definition}

\begin{definition}\label{MDefinition5} Let $f: [l, m] \rightarrow M$ be an admissible transversal. We say that 
  $f$ is {\it{positively (resp. negatively) transverse}} to $ \lambda$ if for every $k \in K$ and every oriented flow box $F$ at $f(k)$ the map (\ref{p1}) is increasing (resp. decreasing) function of s.

Note that all our definitions are clearly seen to be independent of the parameterization. 
\end{definition}

%This follows immediately from the definition of $\nu$ in terms of $\beta$. \begin{definition}\label{MDefinition8} Let $F: R^\sharp=[a, b] \times [0,1]  \rightarrow M$ be a smooth diffeomorphism into its image $R=F(R^\sharp)$. We say that $F=F(t,s)$ is a {\it{smooth homotopy  positively transverse to $ \lambda_u$}} if
%\begin{itemize}
%\item $(i)$$ F([a,b] \times \{ 0,1\})  \subset M_0$
%\item $(ii)$ The family of transversals $s \mapsto f_t(s):=F(t,s)$ depending smoothly on $t$ are positively oriented for all $t$
%\item $(iii)$ $dt \wedge ds = J(F) * 1$, where the Jacobian $J(F) > 0$.
%\end{itemize}
%We define the notion of {\it{smooth homotopy  negatively transverse to $ \lambda_u$}} if we replace condition $(ii)$ with negatively oriented instead of positively oriented.
%\end{definition}
%\begin{remark}
%  It is important to note that we don't require $F$ to be a homotopy in the usual sense i.e $F([a,b] \times  \{s\})$ to be parallel to the leaves of the lamination (cf. Definition~\ref{flowbox}), since we are unable to obtain enough regularity in the $s$-direction for these more restrictive flow boxes (cf. Remark~\ref{nonsmooth}).
%\end{remark}

\begin{definition}\label{MDefinition12} Let $f:[l,m] \rightarrow M$ be an admissible transversal. If $l =l_0 <l_1 < ...< l_n = m$ is a division of $[l,m]$ into intervals on which $f(l_i) \in M_0$ and $f_i(s) = f(s)$ $l_{i-1}\leq s \leq l_i$ is alternatively positively and negatively transverse to $ \lambda$, we say $[l,m] = \bigcup_i [l_{i-1},l_i]$ is a good subdivision for $f$.
\end{definition}
\begin{lemma}\label{MLemma13}  Let $f: [l,m] \rightarrow M$ be an admissible transversal to an oriented lamination $ \lambda$. Then $[l,m]$ has a good subdivision for $f$.
\end{lemma}
\begin{proof} Since there are finitely  many flow boxes we may assume without loss of generality that the image of $f$ is contained in $R$ for some oriented flow box
$F: R^\sharp=[a,b] \times [c,d]  \rightarrow  F(R^\sharp) =R\subset M $. Consider the
continuous map
\[
g=p_2 \circ F^{-1} \circ f: [l, m] \rightarrow [c,d].
\]
Given a point $k \in K$, consider open interval $[k-\epsilon(k), k+\epsilon(k)]$ around $k$ such that $g$ is strictly monotone. By compactness, we can cover $K$ with finitely many such intervals and let $\epsilon=\min \epsilon(k)$.
  
We now construct the good subdivision $l =l_0 <l_1 < ...< l_n = m$. For each $k \in K$ assign  $+=sign(k)$ to the interval $[k-\epsilon(k), k+\epsilon(k)]$ if $f$ defines a positive transversal and $-=sign(k)$ if it defines a negative transversal.
 If $K =\emptyset$, $n = 1$. Let $k_1 = \min_{k \in K} k$ and assign the sign of $sign(k_1)$ to the first interval. Let $k_2 = \min_{k \in K} k$ such that $sign(k_2)$ has the opposite sign. If there is no such $k_2$, $n = 1$. If there is such a $k_2$, choose $l_1 < k_2$ as the largest point less than $k_2$ on for which $f$  is strictly monotone in the interval $[ l_1, k_2]$. Proceed inductively. The process is finite as there is a lower bound $\epsilon$ on the size of the intervals.
\end{proof}

\subsection{Transverse cocycles}\label{tranco} Let $\lambda$ be an oriented geodesic lamination and let  $M_0=M \backslash \lambda$. We write $M_0=\bigcup S$ for {\it finitely many} connected components $S$ called the {\it principal regions} or {\it open plaques} (cf. \cite[Lemma 4.3]{casson}). Lifting to the universal cover we denote $\tilde {M_0}= \tilde M \backslash \tilde \lambda=\bigcup \tilde S$ where  $\tilde S$ is the preimage of $S$. Each component $\tilde S_j$ of $\tilde S$  is also called an open plaque and the projection map $\tilde {S_j} \rightarrow S$ is the universal cover of $S$. Furthermore, the closure of $\tilde {S_j}$ in $H^2$ is a contractible surface with geodesic boundary (cf. \cite[Lemma 4.1]{casson}) and its boundary is contained in the preimage of the boundary leaves of $\lambda$ (cf. \cite[definition and remark on p.61]{casson}).   In this section we start with a map  
   \[
   \tilde v: \tilde M \rightarrow \R
   \]
   with the following properties:
   
   \begin{itemize}
   \item $(i)$ $\tilde v$ is equivariant under a representation $\alpha :\pi_1(M) \rightarrow \R$
   \item $(ii)$ $\tilde v \equiv a_j$ is constant on each plaque $\tilde S_j \subset \tilde M \backslash \tilde \lambda$
   \item $(iii)$ $ \tilde v$ is locally bounded.
   \end{itemize}

 For $S$ and $S'$ open plaques, we set
\begin{equation}\label{MDefinition3}
\beta (S, S') = \tilde v(S)- \tilde v(S').
\end{equation}
Note that since $\tilde v$ is equivariant under  $\alpha$, it follows that $\beta$ is invariant under the action of $\pi_1$. The goal of this section is to define a transverse cocycle $\nu$ induced by $\beta$.

\begin{definition}\label{MDefinition6} For $f :[l,m] \rightarrow M $ an admissible transversal positively oriented, define
$\nu(f) = \beta(S_m, S_l)$ where $\tilde f(l) \in S_l$,  $\tilde f(m) \in S_m$ are the open
plaques containing the endpoints of a lift $\tilde f$. For $f $ an admissible transversal negatively oriented, define
$\nu(f) = -\beta(S_m, S_l)$. Since $\beta$ is invariant under $\pi_1$, this is independent of the lift. 
Finally for an admissible transversal $f$ and  a good subdivision, we define 
\[
\nu(f) = \sum_i  \nu(f_i).
\]  
\end{definition}
\begin{lemma}\label{Mproposition7} $\nu(f)$ does not depend on the choice of $l_i$ in a good subdivision. If $f$ is split into two sub-arcs $f(s) = f_1(s), \ l \leq s  \leq  p$ and $f(s) = f_2(s), \  p \leq s \leq m$ with $f(p) \in M_0$, then $\nu(f) = \nu (f_1) + \nu (f_2)$. Moreover, $\nu$ is invariant under homotopies of $f$  which preserve the lamination and are transverse to the lamination. Finally, $\nu(f) = \nu(f^-)$ where $f^-$ denotes $f$ with the reverse parametrization.
\end{lemma}
\begin{proof}
 We notice that if we choose a second set of $l_i'$, there is an arc between $l_i$ and $l_i'$ which lies in $M_0$. By the properties of $ \nu$, we may move the endpoints of a transversal in $M_0$ without changing $\nu$. Also, the definitions of $\nu$ do not depend on the choice of parameter. Hence the two definitions of $\nu$ agree. The additive property under subdivision of transversals and invariance under change of orientation are immediate  from the definition. 
 
 To see the  invariance under homotopies, consider a  homotopy
 \[
 F: R^\sharp=[a,b] \times [c,d]  \rightarrow  F(R^\sharp) =R\subset M; \ F=F(t,s)
 \]
 and set $f_t=F(t,.)$. 
 Now consider a good subdivision $l =l_0 <l_1 < ...< l_n = m$  of $[l,m]$
  and note that because the homotopy preserves the lamination, the end points
  $f_t(l_i)$ all lie in the same plaque for $t \in [a,b]$.  Therefore,
  \[ 
  \nu(f_t \big |_{[l_i l_{i+1}]})=\nu(f_b \big |_{[l_i l_{i+1}]}).
  \]
  The rest follows from the additive property of $\nu$ with respect to subdivisions.
  \end{proof}
 
 For the next definition, see \cite[page 120]{bonahon2}.
 \begin{definition}\label{deftransversecocycle}
 A {\it{transverse cocycle}} $c$ for an oriented lamination $\lambda$ is a map 
 \[
 c: \{ admissible \  transversals \} \rightarrow \R 
 \]
 which satisfies the following properties: 
 \begin{itemize}
\item $(i)$ $c(f) = c(f_1) + c(f_2)$ when $f$ is decomposed into two subarcs as in Lemma~\ref{Mproposition7}.
\item $(ii)$ $c(f) = c (f')$ when $f$ is carried into $f'$ by a homotopy which preserves $ \lambda$ and is transverse to the foliation.
\item $(iii)$ $c(f) = c(f^-)$ where $f^-$ denotes $f$ with the reverse parametrization.
\end{itemize}
\end{definition}

Lemma~\ref{Mproposition7} now implies immediately:

\begin{theorem}\label{thm:tranco}A function $\tilde v$ satisfying properties $(i)$-$(iii)$ defines a transverse cocycle $\nu$. 
\end{theorem}

The following is Thurston's definition of transverse measure, more or less in Thurston's own words. (See \cite[Section 8.6]{thurston2}.)
\begin{definition}A transverse measure $\nu$ for a geodesic lamination $\lambda$ means a measure  defined on each local leaf space $[c,d]$ of every flow box,
in such a way that the coordinate changes are measure preserving. Alternatively one
may think of $\nu$ as a measure defined on every admissible (unoriented) transversal 
to $\lambda$, supported on the intersection of the transversal with the lamination and invariant under local projections along leaves of $\lambda$. 

In this paper we use this definition except we allow the support of the measure to possibly be strictly contained in the intersection of the transversal with the lamination. It is straightforward that a transverse cocycle $c$ is  a transverse measure iff $c(f) \geq 0$ for every $f$ positively transverse to $\lambda$ (cf. \cite[Proposition 18]{bonahon2}.)
\end{definition}

 \subsection{The transverse measure on $\lambda_u$} We now go back to the sequence
$v_q $ of $q$-harmonic sections  converging  as in Theorem~\ref{lemma:limmeasures2} and Theorem~\ref{thmlegr} to a fixed least gradient section $v$ along a sequence $q \rightarrow 1$. Also,   
$\lambda_u$  is the geodesic lamination of maximum stretch of the $\infty$-harmonic map $u$ constructed in Section~\ref{sect:crandal}. 
The main theorem of the section is:

\begin{theorem}\label{transmeasure}The least gradient map $v$ induces a transverse measure $\nu$ on the geodesic lamination $\lambda_u$. 
\end{theorem}

Let   $\sigma: \tilde M \rightarrow M$ denote the universal cover, and denote the lift of $v_q$ by $\tilde v_q$, the lift of $v$ by $\tilde v$ and so forth.
 Let $M \backslash  \lambda_u = M_0$ and $\tilde M_0 = \sigma^{-1}(M_0)$. 
By Theorem~\ref{straightline},
\begin{equation}\label{MLemma1}
M_0 = |du|^{-1}([0,L)) \ \mbox{and} \ \tilde M_0 = |d\tilde u|^{-1}([0,L).
\end{equation}
The lamination $\lambda_u$ has in our context a natural orientation given by $grad\ u$.
Let $\tilde M_0 =\bigcup S_j$ where $S_j$ are the open connected components of $\tilde M_0$.
 \begin{lemma}\label{MLemma2} $\tilde v(x) = a_j$ is constant for $x$ in the open plaque $S_j$ and the constants $a_j$ are locally bounded in $\tilde M$.
Moreover, the sequence $\tilde v_{q_j}$ converges to the constant $a_j$ in 
$W^{1,1}_{loc}(S_j)$. 
\end{lemma}
\begin{proof}
Let $\tilde B$ in $S_j$ be a closed ball in $S_j$ and let $\chi_{\tilde B}$ denote its characteristic function. From Proposition~\ref{prop:supptmeasure0}, by (\ref{normintv1}), (\ref{normintv2}) and Lemma~\ref{klemma1}, 
\[
\lim_{q \rightarrow 1} \int_{\tilde B} |d \tilde v_q|^q *1 =
L^{-1} \lim_{q \rightarrow 1} \int |U_p|^p \chi_{\tilde B}*1 = 0.
\]
By combining with Theorem~\ref{lemma:limmeasures2}, the  $\tilde v_q $ converge to $\tilde v$ in $W^{1,1}_{loc}(\tilde B)$ and also in $L^s_{loc}(\tilde B)$ for all $s$ where $d\tilde v=0$. Thus,  $\tilde v=a_j$ in $S_j$.
\end{proof}

\begin{transmeasure}
It suffices to show $\nu$ is non-negative on positive transversals. Let 
  $F: R^\sharp=[a, b] \times [0,1]  \rightarrow M$ a smooth map such that $f_t (s) = F(t,s)$  are positively transverse to to $\lambda_u$ and $f=f_b$. Notice that we don't require $F$ to be a flow box as we cannot simultaneously assume that $F$ is smooth. See Remark~\ref{nonsmooth}.
  %  -----------------------------------
%  Then $\nu(f_t)$ is independent of $t$ and $\nu =
%\nu(f_t ) \geq 0$.
%
% The fact that $\nu(f_t)$ is independent of $t$ follows from Lemma~\ref{Mproposition7}. \\
%----------------------------------------- 
 We will use the fact that $v_q \rightarrow v$ as $q \rightarrow 1$ in $L^1_{loc}$. Also, the image of $F$ is simply connected and we may choose real valued representatives of $v_q$, $u_p$, $v$ and $u$  rather than working in the cover. 
%Choose $c$ so that:
%\begin{itemize}
%\item $(i)$ $F( [a,b] \times [0,c]) = R_0$ in $M \backslash  \lambda_u$
%\item$(ii)$ $v$ is identically $a_0$ on $R_0$
%\item$(iii)$ $F( [a,b] \times [1-c,1])=R_1$ in $M \backslash  \lambda_u$
%\item$(iv)$ $v$ is identically $a_1$ on $R_1$.
%\end{itemize}
%Our goal is to show that $\nu =a_1 - a_ 0 \geq 0$. 
%By the chain rule
%\[
%\frac{d}{ds}v_q(F(t,s)) = (dv_q)_{F(t,s)}\left(\frac{dF}{ds}\right).
%\]

Choose $R_0^\sharp = [a,b] \times [0,c]$, $R_1^\sharp=[a,b]  \times [1-c,1]$ so that $F(R_i^\sharp) \subset M_0$ and $v$ is constant equal to $a_i$ on $R_i=F(R_i^\sharp)$. We have to show $a_1-a_0 \geq 0$. It will be convenient in the computations to write  $ v_q\circ F= v^\sharp_q$ and similarly for any other function defined on a subset of $R$. Choose a non negative cut-off function $\xi^\sharp \in C^\infty_0(R^\sharp)$ such that $\xi^\sharp (t,s) = \xi^\sharp (t, 1-s)$ and
\begin{eqnarray}\label{Ksi1}
\mu:=\int_{R_i^\sharp}\xi^\sharp(t,s)dtds > 0.
\end{eqnarray}
By the chain rule, 
\[
\frac{dv^\sharp_q}{ds}  = (d v_q)\circ F \frac{dF}{ds},
\]
hence by integrating in  $s$, we get
\[
 v_q^\sharp(t,1-\tau) -  v_q^\sharp(t,\tau) = \int_\tau^{1 - \tau} (dv_q)\circ F \left( \frac{dF}{ds} \right)ds.
\]
%\[
%\int_a^b v_q^\sharp(t,1-\tau)dt - \int_a^b v_q^\sharp(t,\tau)dt = \int_a^b\int_\tau^{1 - \tau} (dv_q)\circ F \left( \frac{dF}{ds} \right)dsdt.
%\]
Now multiply by $\xi^\sharp$ and integrate in $t$ from $a$ to $b$ and in $\tau$ from 0 to $c$ to obtain
\begin{eqnarray}\label{Ksi2}
\lefteqn {\int_{R_1^\sharp} \xi^\sharp(t,s)v^\sharp_q(t,s) dtds - \int_{R_0^\sharp} \xi^\sharp(t,s)v^\sharp_q(t,s)   dtds} \nonumber \\
&=&\int_a^b\int_0^c\int_\tau^{1 - \tau} (dv^\sharp_q)_{F(t,s)}\left(\frac{dF(t,s)}{ds}\right)ds\xi^\sharp(t,\tau)d\tau dt\\
&=&\int_{R^\sharp} (dv^\sharp_q)_{F(t,s)}\left(\frac{dF(t,s)}{ds}\right) \Xi^\sharp(t,s) dt ds. \nonumber
\end{eqnarray}
Here the positive function $\Xi^\sharp (t,s)$ can be explicitly computed from interchanging integration in $s$ and $\tau$. For $s < 1/2$
\[ 
\Xi^\sharp(t,s) = \int_0^{\min(s,c)} \xi^\sharp(t,\tau) d\tau= \Xi^\sharp(t,1-s).
\]
We will not use the explicit formula, however note that $\Xi^\sharp$ has compact support in the interior of $R^\sharp$ and hence $\Xi= \Xi^\sharp \circ F^{-1}$ has compact support in the interior of $R$.
By using  (\ref{normintv2}), (\ref{Ksi2}) implies
\begin{eqnarray*}
\lefteqn {\int_{R_1^\sharp} \xi^\sharp(t,s)v^\sharp_q(t,s) dtds - \int_{R_0^\sharp} \xi^\sharp(t,s)v^\sharp_q(t,s)   dtds}\\
&=&\int_{R} dv_q\left(\frac{dF}{ds} \circ F^{-1} \right)  \Xi  J(F^{-1})*1\\
&= &  \int_{R} |U_p|^{p-2} *U_p \left(\frac{dF}{ds}\circ F^{-1} \right)\Xi  J(F^{-1})*1.
\end{eqnarray*}
By (\ref{Ksi1}) and the fact that $ v_q \rightarrow  v$ in $L^1_{loc}$, the left-hand side has the limit $\mu(a_1-a_2)$.  
By Proposition~\ref{kprop 4}  and Lemma~\ref{klemma1}, 
\begin{eqnarray*}
&&\lim_{p \rightarrow \infty}\int_{R}|U_p|^{p-2}*U_p \left(\frac{dF}{ds}\right)(\Xi \circ F^{-1})J(F)*1\\
&=& L^{-1}\lim_{p \rightarrow \infty}\int_{R}|U_p|^{p-2}*du_p \left(\frac{dF}{ds}\right)(\Xi \circ F^{-1})J(F)*1.
\end{eqnarray*}
However, in the definition of positively transverse $*du_p \left(\frac{dF}{ds}\right) > 0$,  the Jacobian $J(F)>0$ and $\Xi \geq 0$ with $\Xi > 0$ on a set of positive measure. So the right hand side is the limit of positive numbers; hence the limit must be non-negative.
By comparing with the left hand side, we obtain that  $a_1 - a_0 \geq 0$. 
\end{transmeasure}

  \begin{remark}We do not claim that the limit is positive. There can be leaves of $ \lambda_u$ on which the transverse measure vanishes.
\end{remark}
  
  We end the section by proving a general theorem relating the notion of transverse cocycles with functions of bounded variation in the case when the cocycle is non-negative. More precisely, we show:

  \begin{theorem}\label{transcoismeasthm} Assume $\lambda$ is an oriented geodesic lamination and $\tilde v: \tilde M \rightarrow \R$ satisfies properties $(i)$-$(iii)$ as in Section~\ref{tranco}. If the transverse cocycle $\nu$ associated to $\tilde v$ via Theorem~\ref{thm:tranco} is a transverse measure, then $\tilde v$ is locally of bounded variation.
  \end{theorem}

%  The local boundedness of $\tilde v$ implies that $\tilde v \in L^1_{loc}(M)$. We are next going to show that the measure $d\tilde v$ has bounded variation.
\begin{proof} 
  Since the problem is local we will work locally in $M$ instead of $\tilde M$ and consider $v$ instead of $\tilde v$. Let
   $F=F(t,s): R^\sharp=[a,b] \times [c,d]  \rightarrow  R \subset M$ be
 a flow box as in (\ref{eqn:flowbox}), set $f_t(s) = F(t,s)$ and consider the fixed transversal $f = f_b$. By definition,  $f_b$ is  positively oriented with respect to the oriented lamination $\lambda$ and since $\nu$ is non-negative by assumption, the function 
\begin{eqnarray}\label{gsharp}
g^\sharp(s) &: =&\int_c^sf^*( d\nu)
 =\int_{\{b\} \times [c,s]} d\nu \nonumber \\
& =& \int_c^s d\nu \ \ \mbox{(by a slight abuse of notation)} 
\end{eqnarray}
is non-decreasing. Furthermore,
\begin{equation}\label{formvsharp}
v^\sharp(t,s)=v^\sharp(t,c)+g^\sharp(s).
\end{equation}
In order to show (\ref{formvsharp}), assume that $F(t,c)$ is in the plaque $S_0$ and  $v^\sharp(t,c)=v^\sharp(b,c)=a_0$ and $F(t,s)$ is in the plaque $S$ and  $v^\sharp(t,s)=v^\sharp(b,s)=a$. Since the transversal $f=f_b$ is positively oriented with respect to the lamination, we have by Definition~\ref{MDefinition6} that $g^\sharp(s)=\beta(S,S_0)=a-a_0$. Hence  (\ref{formvsharp}) follows.

Since the measure $\nu$ is positive, the function $g^\sharp$ is monotone and hence of bounded variation. Formula ~(\ref{formvsharp}) then implies that $v^\sharp$ is of bounded variation. Since $\tilde v=v^\sharp \circ F^{-1}$, \cite[Theorem 3.16]{ambrosio} implies that $\tilde v$ is locally of bounded variation with and  $|d \tilde v| \leq F_*|dv^\sharp|$ locally.
\end{proof}

\section{From transverse measures to functions of bounded variation}\label{ruelles}
In the previous sections we showed that, given an oriented geodesic lamination $\lambda$ in a hyperbolic surface $M$ and  a locally bounded function $v$, which  is  constant on the plaques of $M_0=M \backslash \lambda$, we can construct a transverse cocycle $\nu$. Moreover, if $\nu$ is non-negative, then $\nu$ is a transverse measure and this forces $v$ to be of bounded variation. In this section we will start with a transverse measure $\nu$ on $\lambda$ and we will construct  $v$ as a primitive of BV to the Ruelle-Sullivan current. We continue to assume throughout the section that $M$ is a closed hyperbolic surface.

 \subsection{The Ruelle-Sullivan current} 
 
 In 1975 Ruelle-Sullivan \cite{sullivan} constructed a current for a transverse measure on a partial foliation. The next construction follows theirs (with less regularity for $F$), but we repeat it for completeness.

 \begin{definition}\label{integration} Let $\Lambda=(\lambda, \nu)$ be an oriented measured geodesic lamination and $F_i: R_i^\sharp=[a_i,b_i] \times [c_i,d_i] \rightarrow R_i= F_i(R_i^\sharp) \subset M$ be flow boxes as in Definition~\ref{flowbox} covering a neighborhood $\mathcal U$ of $\lambda$.
 Define an 1-current
 $T_\Lambda$  by setting
 \[
 T_\Lambda(\phi)=\sum_i \int_{(c_i,d_i)}\left(\int_{[a_i,b_i] \times \{ s\}}F_i^*(\phi_i) \right)d\nu(s); \ \ \phi=\sum_i \phi_i
 \]
 where $\phi \in \mathcal D^1(\mathcal U)$ and $\phi_i \in \mathcal D^1(R_i)$.
 \end{definition}
 
 \begin{theorem}\label{welldefcur} $T_\Lambda$ is a well defined 1-current. Furthermore, $T_{\Lambda} $ is closed and thus defines an element
 \[
 [T_{\Lambda}] \in H_1(M, \R).
 \]
 \end{theorem}
 \begin{proof}
 First, note that because $\frac{\partial F_i}{\partial t}$ is continuous, 
 \[
 s \mapsto \int_{[a_i,b_i] \times \{ s\}}F_i^*(\phi)= \int_{a_i}^{b_i} ({F_i}_s)^*\phi
 \]
 is a continuous function in $s$, so
 we can integrate against a Radon measure. To show it is independent of the choice of flow box, first consider the case where $ \phi$ is compactly supported in the intersection of two flow boxes $F$ and $F'$.
 Then,
 \begin{eqnarray*}
 \int_c^d\int_a^bF_s^*(\phi)d\nu(s)&=&\int_c^d\int_a^b({F'}^{-1}F)_s^*{F'}_s^*(\phi)d\nu(s)\\
 &=&\int_{a'}^{b'}\int_{c'}^{d'}{F'}_s^*(\phi)({F}^{-1}F')_s^*(d\nu(s))\\
 &=&\int_{a'}^{b'}\int_{c'}^{d'}{F'}_s^*(\phi)d\nu(s),\\
 \end{eqnarray*}
 the last equality because the transverse measure is invariant under the transition functions ${F}^{-1}F'$.
 We can reduce the general case to this case, as follows: Consider two atlases consisting of flow boxes $\{F_i\}$ and $\{F'_{i'}\}$ and let  $\{\xi_i\}$ and $\{\xi_{i'}\}$ be partitions of unity subordinate to the above covers. We can write 
 \[
 \phi=\sum_{i,i'}\xi_i\xi_{i'} \phi.
 \]
By the previous case,
 \[
 \int_c^d\int_a^bF_s^*(\xi_i\xi_{i'}\phi)d\nu(s)=\int_{a'}^{b'}\int_{c'}^{d'}{F'}_s^*(\xi_i\xi_{i'}\phi)d\nu(s),
 \]
 thus by summing over $i,i'$ we obtain the desired equality.
 To show it is closed note that if $f \in \mathcal D^0(R_i)$ is supported in one flow box, 
 \[
 \int_{[a_i,b_i] \times \{ s\}}F_i^*(df)=\int_{a_i}^{b_i} \frac{\partial}{\partial t} (f \circ F_i)(t,s)dt=0
 \]
 hence
 \[
 T_\Lambda (df)=0.
 \] 
 \end{proof}
 
% \begin{theorem} \label{equalcurrents} If $\Lambda_u= (\lambda_u, \nu)$ is the oriented measured lamination of Theorem~\ref{transmeasure}, then as currents
%\[
%T_{\Lambda_u}=dv.
%\]
%\end{theorem}
%\begin{proof} {\bf{To be added from the previous version}}.
%\end{proof}

\subsection{Constructing a primitive of the Ruelle-Sullivan current} 
In Theorem~\ref{transcoismeasthm}, we showed that a cocycle which defines a transverse measure is associated to a function of bounded variation.  For example, if the cocycle is non-negative, then by a theorem of Bonahon it is a transverse measure.  We now show directly that the cocycle of Bonahon and the corresponding function of bounded variation can be constructed directly from the Ruelle-Sullivan current.

\begin{theorem}\label{conversetomeasure} Given an oriented measured geodesic  lamination $\Lambda=(\lambda, \nu)$, there exists a flat real affine rank 1 bundle  $L$ and a section $v: M \rightarrow L$ of bounded variation  such that
\begin{equation}\label{T=dv}
T_\Lambda=dv.
\end{equation}
\end{theorem}

\begin{proof}

Let $F=F(t,s): R^\sharp=[a,b] \times [c,d]  \rightarrow  R \subset M$
be a flow-box as in (\ref{eqn:flowbox}). 
First we are going to consider the local problem 
and construct $v=v_F$ on the image  of $F$.
Let $f_t(s) = F(t,s)$ and consider the fixed transversal $f = f_b$.  
Consider the non-decreasing function $g^\sharp(s)$ defined as in (\ref{gsharp})
by
$g^\sharp(s) =\int_c^sf^*( d\nu)=
 \int_c^s d\nu $
%We also parameterize the $t$ variable so that $|\frac{d \lambda_0}{dt} | = 1$ for all leaves $\lambda_0$ in $\lambda$ which intersect $R$.
%Let 
%\[
%K = \{k \in [c,d]: f(k) \in \lambda \}.
%\]
% $K$ is a closed set of measure 0 and Haussdorff dimension 0, and $g^\sharp$ is constant on the connected components of $[c,d] \backslash K$. Denote the leaf of $\lambda$ which passes through $k \in K$ by $\lambda_k$.
%Throughout this section the maps $F$ and $f_t $ identify functions on $[a,b] \times [c,d]$ or a subset with functions on $R$. As in the previous sections, we use $g^\sharp = f^*g$, $v^\sharp = F^*v$
%and so forth consistently. Remember, however, that $F$ is only Lipschitz although we may assume $f $ is smooth.
%Let
and let
\begin{equation}\label{vsharp}
v^\sharp(t,s) = g^\sharp(s). 
\end{equation}
We define
$v_F:=v^\sharp \circ F^{-1}$ in the image of the flow box $F$ in terms of $g^\sharp$ as above. 
Note that by the invariance of $\nu$ under homotopies, $v_F$
is constant on the plaques in the image of $F$. Also
 $v_F$ is  bounded. Furthermore, $v^\sharp$ is of bounded variation, because $g^\sharp$ is monotone and  
 as in the proof of Theorem~\ref{transcoismeasthm}, we can conclude that $ v_F$ is of bounded variation.
 
 Note that for compactly supported $\phi^\sharp=\phi_1dt+\phi_2ds \in \mathcal D^1(R)$,  
 \begin{eqnarray*}\label{dvsharp1}
  \lefteqn{\int_{(c,d)}\left(\int_{[a,b] \times \{ s\}}\phi^\sharp \right)d\nu(s)} \nonumber\\
 &=&- \int_{(c,d)}\frac{d}{ds}\left(\int_{[a,b] \times \{ s\}}\phi^\sharp \right) g^\sharp ds \ \ (\mbox{by definition of} \ g^\sharp) \nonumber\\
 &=& -\int_c^d\left(\int_a^b \frac{\partial}{\partial s}\phi_1(t,s)dt \right) g^\sharp(s) ds \nonumber\\
 &=& -\int_c^d\left(\int_a^b \frac{\partial}{\partial s}\phi_1(t,s)dt+ \frac{\partial}{\partial t}\phi_2(t,s)dt \right) g^\sharp(s) ds 
 \ \ (\phi_2 \ \mbox{compactly supported}) \\
 &=& -\int_c^d \left(\int_a^b \frac{\partial}{\partial s}\phi_1(t,s)dt+ \frac{\partial}{\partial t}\phi_2(t,s)dt \right)  v^\sharp(t,s)dt ds \nonumber\\
 &=& \int_c^d \int_a^b v^\sharp d \phi^\sharp  dt ds \nonumber\\
 &=& dv^\sharp(\phi^\sharp). \nonumber
  \end{eqnarray*}
 
 By Definition~\ref{integration}, this implies that in the interior of a flow box $F$ (\ref{T=dv}) holds.
 If $F: R^\sharp \rightarrow M$ and $F': {R'}^\sharp \rightarrow M$ are two flow boxes  which intersect in a non-empty connected set which contains a ball, then on the overlap
 \begin{equation}\label{globaltheory}
 dv_F = dv_F'=T_\Lambda
 \end{equation}
 hence 
$v_F = v_F' + c.$

We now proceed with constructing a flat affine  line bundle $L$ over the surface $M$ and a global section $v: M \rightarrow L$ formed by piecing together the local primitives of the Ruelle-Sullivan current $T=T_\Lambda$.
 
Choose a smoothing $T^\epsilon$ of $T$,
 $d (T^\epsilon)=(d T)^\epsilon=0$
and let
 $\tilde T^\epsilon =\sigma^*(T^\epsilon)$ be the pullback to the universal cover. Let  $\tilde v^\epsilon$ be a primitive of  $\tilde T^\epsilon $ equivariant under representations
 $\alpha^\epsilon: \pi_1(M) \rightarrow \R.$
 By the Poincare Lemma, we can write
 $\tilde T^\epsilon=   d\tilde v^\epsilon$,
 where $v^\epsilon$ is a smooth real valued function equivariant under the representation 
 $\alpha^\epsilon$. As is Section~\ref{qgoesto1},
 $\alpha^\epsilon \rightarrow \alpha$ and
 the convergence as distributions $d\tilde v^\epsilon=\tilde T^\epsilon \rightarrow T$ and weak compactness in the space BV, implies that up to a constant $\tilde v^\epsilon \rightharpoonup \tilde v$  in $BV_{loc}(\tilde M)$ where $\tilde v$ is equivariant under $\alpha$.
  If $L$ denotes the flat affine line bundle associated to 
 $\alpha$, then $v$ is a section of $L$ of bounded variation.
\end{proof}

\begin{corollary}Let $\nu$ denote the measure on the lamination $\lambda_u$ constructed in Theorem~\ref{transmeasure} associated to the least gradient map $v$. If $\Lambda_u=(\lambda_u, \nu)$, then the Ruelle-Sullivan current $T_{\Lambda_u}=dv$.
\end{corollary}

\begin{proof}The measure $\nu$ is related with the least gradient map $v$ by formulas (\ref{gsharp}) and (\ref{formvsharp}). The Ruelle-Sullivan current associated to $\nu$ is given by the derivative of a new function $v$ given by (\ref{vsharp}) which only differs from (\ref{formvsharp}) by a constant. Thus $T_{\Lambda_u}=dv$. 
\end{proof}

\subsection{The decomposition in terms of functions of bounded variation}
In the previous sections we showed how transverse measures  correspond to  functions of bounded variation. In this section we will explore how different types of leaves of the lamination correspond to different types of functions of bounded variation. 

First
recall  that
a leaf  $\lambda_0$ of $\lambda$ is called {\it{isolated}}  if for each $x \in \lambda_0$ there exists a neighborhood $U$ of $x$ such that $(U, U \cap \lambda_0)$  is homeomorphic to (disc, diameter) \cite[Definition p.46]{casson}.  
 A  geodesic lamination $\lambda$ is called {\it{minimal}}, if it is minimal with respect to inclusion.

\begin{theorem}\label{cassonmin}\cite[Theorem 4.7 and Corollary 4.7.2]{casson}  A geodesic lamination is minimal iff each leaf is dense. Any geodesic lamination is the union of finitely many minimal sub-laminations and of finitely many infinite isolated leaves, whose ends spiral along the minimal sub-laminations.
\end{theorem}

%------------------------------
%There is a bit say about the analysis interpretation of transverse measures. 
%%Recall that $\lambda_u$ was an oriented geodesic lamination, and $\nu$  a transverse measure for $\lambda_u$.
%The next discussion applies to any oriented geodesic lamination $\lambda$ with a transverse measure $\nu$. 
%Let $\tilde \lambda$ denote the lift of $\lambda$ to the universal cover $\sigma: \tilde M \rightarrow M$. 
% From the definition of transverse measure, we see immediately that if $f$ is a transversal that intersects $\lambda$, in a point  on an isolated once,  any transverse measure on $f$ will put a non negative delta function at the point of intersection. By the homotopy invariance of a transverse measure, this is constant along $\lambda_0$. Define this to be the jump at $\lambda_0$.
In the presence  of a transverse measure $\nu$ one can easily characterize isolated leaves by using the homotopy invariance of the measure. 
If a leaf is closed, then $\nu$ is an atomic measure i.e a delta function supported at the point of intersection of the transversal and the corresponding function of bounded variation $v$ is a jump function. On the other hand, a spiraling isolated leaf $\lambda_0$ cannot support a non-zero measure because a transversal crossing $\lambda_0$  at a limit point  would have infinite measure, as it crosses $\lambda_0$ an infinite number of times.

In view of the above, from now on we will assume that the oriented lamination $\lambda$ is a disjoint union of finitely many minimal sub-laminations,  and
 we will  explore the dichotomy between closed leaves and infinite non-isolated leaves in terms of functions of bounded variation.

Recall that there are three types of functions of bounded variation defined on a ball in a Riemannian manifold:\\ 
$(i)$ Functions in the Sobolev class $W^{1,1}$.\\
$(ii)$ Jump functions across a 
(countably) rectifiable set of codimension 1.\\
 $(iii)$ Cantor functions, which are
continuous functions with derivative zero on a dense open set.
  A nice description of Cantor functions can be
found in \cite{cantorexp}.

In fact, the derivative of any function of bounded variation can be decomposed into these three types, according to the following theorem
(cf. \cite[Theorem 3.78 and Proposition 3.92]{ambrosio}).

\begin{theorem}\label{decocantor1}
Let  $v: B \rightarrow \R$  be a function of bounded variation  defined in a ball $B$. Then, there is a canonical decomposition  
\[
dv = (dv)_0 + (dv)_{jump} + (dv)_{cantor}
\]
where: 
\begin{itemize}
 \item $(i)$ The measure $(dv)_0$ is absolutely continuous with respect the Lebesgue measure
and the measure $(dv)_{jump} + (dv)_{cantor}$ is singular with respect to Lebesgue measure.
\item $(ii)$ $(dv)_{cantor}$  vanishes on  every Borel set  with $\sigma$-finite $n-1$ Hausdorff measure.
\item $(iii)$   $(dv)_{jump}$ is computed as a measure of a jump discontinuity on a countably $n-1$ dimensional rectifiable set.
\end{itemize}
\end{theorem}

\begin{corollary}\label{decocantor2} Let $v$ be the  primitive of the Ruelle-Sullivan current  associated to an oriented measured geodesic lamination
$\Lambda=(\lambda, \nu)$. In the decomposition of the measure $dv$, 
\begin{itemize}
\item $(i)$ $(dv)_0 = 0$.
\item $(ii)$ $(dv)_{jump}$ is supported on closed isolated leaves. 
\item $(iii)$ $(dv)_{cantor}$ is supported on minimal laminations which are not closed isolated leaves.
\end{itemize}
\end{corollary}

%% BV functions have a standard decomposition.  Here we can work in a ball $B$ in $M$ and use a local definition of $v: B \rightarrow \R$, since it is not a global theorem.  A BV function defined in a neighborhood of $B$  has a canonical decomposition $v = v_0 + v_{jump} + v_{cantor}$, where $v_0$ in $W^{1,1}(B)$, $v_{jump}$ is locally constant with jumps along rectifiable sets of codimension 1 and $v_{cantor}$ is continuous with derivative zero on an open dense set.  See  \cite[Theorem 3.77 and Proposition 3.92]{ambrosio} and \cite{evans}  for precise definitions and details.
%%
%%George:  I can?t get those books off the web. Possibly available from the Princeton University library when I get back. I do not understand the u notes Camillo sent to me, but they are also unpublished (not on the web)  I will work at getting the statement of this theorem precise.  Perhaps we don?t have to define everything exactly.  By the way, the theorem is true for vector-valued functions and there is something about a rank 1 theorem.  Might turn out to be very useful mapping surfaces to surfaces.

%\begin{theorem}\label{Mtheorem19} The canonical division of $v$ obtained as in section 4 is $v = v_{jump} + v_{cantor}$, where the support of $dv_{jump}$ is on the closed geodesics $\lambda_ {periodic}$ and $dv_{cantor}$ is supported on $\lambda - \lambda_{0 periodic}$.
%\end{theorem}
\begin{proof}
The support of $dv$ lies on the lamination, which is of measure 0, so 
$(dv)_0= 0$.
  By the decomposition Theorem~\ref{cassonmin}, we can write  $\lambda$
 as a disjoint union of minimal sublaminations $\lambda_1$ which is the disjoint union of  closed leaves and $\lambda_2$ which is the disjoint union of all sublaminations consisting  of minimal non-isolated leaves. Given a flow box $F$, there is a finite number of points in $K=\{k_1,...,k_n \}$ corresponding to closed isolated leaves  and write $c=k_0< k_1 <... < k_{n+1} =d$. These coincide with the atoms of the measure $\nu$, hence by \cite[Corollary 3.33]{ambrosio} the discontinuity set of $dg^\sharp=d\nu$ must be equal to the set 
 $\{ k_1,..., k_n\}$. Since $v^\sharp(s,t)=g^\sharp(s)$, it follows that the discontinuity set of $dv$ is equal to $\lambda_1$. By 
 \cite[Definition 3.91]{ambrosio}, $(dv)_{jump}$ is supported on $\lambda_1$ and $(dv)_{cantor}=dv-(dv)_{jump}$ on $\lambda_2$.
\end{proof}

Actually we can prove a stronger statement

\begin{proposition} If $v$ is the primitive for the Ruelle-Sullivan current associated to an oriented measured geodesic lamination $\Lambda$, then 
\[
v = v_1 + v_2
\]
 and $dv_1=(dv)_{jump}$, $dv_2=(dv)_{cantor}$. Here $v_1$ is a jump function and $v_2$ is a Cantor function.
\end{proposition}
\begin{proof}
% The support of $(dv)_{jump}$ is on closed geodesics and the support of $(dv)_{cantor}$ is on the union of minimal laminations in which every leaf is dense. By the decomposition Theorem~\ref{cassonmin}, these are disjoint measured laminations. Let $v_{jump}$ be a primitive for $(dv)_{jump}$ and $v_{cantor}$ for $(dv)_{cantor}$. Hence the proposition follows since the primitives are unique (up to an additive constant).
% \end{proof}
% 
% -----------------
% \begin{proof}The support of $dv$ lies on the lamination $\lambda$, which is of measure 0, so 
%$(dv)_0= 0$.
%  By the decomposition Theorem~\ref{cassonmin} we can write  $\Lambda$
% as a disjoint union of minimal sublaminations $\Lambda_1$ which is the union of sublaminations consisting of closed leaves and $\Lambda_2$ which is the union of all sublaminations consisting  of minimal non-isolated leaves. The lamination $\Lambda_1$ is a set which is 1-rectifiable, whereas  $\Lambda_2$ is purely unrectifiable. Hence 
% \[
% (dv)_{jump}=dv \Big |_{\Lambda_1} and \ (dv)_{cantor}=dv \Big |_{\Lambda_2}.
% \]
% 
Let  $\Lambda_i=(\lambda_i, \nu)$ where $\lambda_i$ as in the previous Corollary and $T_{\Lambda_i}$  the  Ruelle-Sullivan currents corresponding to $\Lambda_i$.  Then,
 $T_{\Lambda_i}$ is closed and let $v_i$ be their primitives, $dv_i=T_{\Lambda_i}$ as in Theorem~\ref{conversetomeasure}. Since
 $T_{\Lambda}=T_{\Lambda_1}+T_{\Lambda_2}$ and since the primitives are unique (up to an additive constant),
 $v=v_1+v_2$. By construction, $dv_i$ is supported on $\lambda_i$ and thus $dv_1=(dv)_{jump}$ and $dv_2=(dv)_{cantor}$.
 % With the notation as on Proposition~\ref{discrleav2}, in a given flow-box, there is a finite number of points in $K=\{k_1,...,k_n \}$ corresponding to closed isolated leaves  and write $c=k_0< k_1 <... < k_{n+1} =d$. The function $g^\sharp=a_j$ is constant in the open intervals $I_j =(k_j, k_{j+1})$,  so
% $v = a_j$ constant on the rectangular shaped region $R_j=F([a,b] \times I_j)$ bordered by the two end transversals and the geodesics $\lambda_j$ and $\lambda_{j+1}$ which pass through $k_j$ and $k_{j+1}$. Thus $v_1=:v_{jump}$  is a jump function and $dv_{jump}$ is a jump discontinuity measure. 
% 
% Notice that 
%  
% Note that $v_2$ is constant on the open plaques in $M_0$ corresponding to the minimal, non-isolated components of the lamination, hence $v_2=v_{cantor}$ must be a Cantor function.
 \end{proof}
% 
%  \begin{lemma} Suppose that the support of the Ruelle-Sullivan current $T$ is on a set $S = S_1 \cup S_2$ where $S_1 \cap S_2 = \emptyset$. Then $T=T_1+T_2$, where $T_j$ has support on $S_j$ and $v = v_1 + v_2$, where $dv_j = T_j$.
% \end{lemma}
%\begin{proof} Let $T_j = \phi_j T$ where $\phi_1$ is a smooth function such that $\phi_1(x) = 1$ for $x \in S_1$ and $\phi_1(x) = 0$ for $x \in S_2$ and vice versa for $\phi_2$. Then $T_j$ is closed and the primitive for $T_j$ is constructed as in Theorem~\ref{conversetomeasure}. Then $v = v_1 + v_2$ has the requisite properties since $v$ is unique up to a constant.
%\end{proof}

 This also applies to the transverse least gradient measures obtained from best Lipschitz maps. Note that the cohomology classes associated with the transverse measures and laminations add.

\begin{corollary} Let $\Lambda=(\lambda, \nu)$ be an oriented lamination without isolated leaves. Then the primitive $v$ of the Ruelle-Sullivan current $T_\Lambda$ defines a continuous but not absolutely continuous section $v: M \rightarrow L$ whose derivative is zero almost everywhere.
\end{corollary}

\section{Conjectures and Open Problems}\label{conjectures}
As mentioned already the authors' motivation for this paper is understanding Thurston's work of best Lipschitz maps between surfaces. As such, the results of this paper only serve as a toy problem in understanding the more difficult problem of best Lipschitz maps between surfaces.
This paper is by no means complete and is only meant to be the preliminary part of a more thorough study.   
This section contains some  suggestions for new directions for research. 

The main topic of study in this paper is $\infty$-harmonic maps from hyperbolic manifolds to $S^1$ and their maximum stretch laminations. The theory of $\infty$-harmonic functions has been thoroughly worked out for Euclidean metrics but so far no work has been done for variable metrics. For example, there is no reference of viscosity solutions for other than flat metrics and there is no reference for the equivalence with the notion of comparison with cones. In Section~\ref{sect:crandal} (cf. Proposition~\ref{comocones1}), we worked out the bare minimum of what we needed from the theory of comparison with cones in order to obtain our results on  geodesic laminations. However, the theory  is far from complete. In particular, for the sake of simplicity, we only considered the hyperbolic metric, thus leaving the theory for general metrics as a conjecture: 

\begin{conjecture} Develop the theory of $\infty$-harmonic functions for  general  Riemannian metrics. Most of the known results about $\infty$-harmonic functions (including the theorems of Crandall on gradients \cite{crandal} and the regularity results of  Evans-Savin \cite{evans-savin} and Evans-Smart \cite{evans-smart}) should carry over to this case.
 In particular, show that
 Theorem~\ref{straightline} holds for any Riemannian manifold $(M,g)$.
\end{conjecture}

There are two uniqueness theorems which we believe to be true, but cannot prove.
\begin{conjecture}\label{Conjecture 2} The $\infty$-harmonic map $u: M \rightarrow S^1$ in a homotopy class is unique up to rotation in $S^1$.
\end{conjecture}

%Given two such maps $u$ and $u'$, one can consider their difference in the universal cover $\tilde u'- \tilde u : \tilde M 
%\rightarrow \R$. One can try to use maximum principle arguments, but the problem is that the maximum might well be obtained on a closed set which is not homotopically trivial and has no neighborhood diffeomorphic to an open domain in the plane. This conjecture turns out not to be a serious deficiency, as Theorem~\ref{K=L} gives another characterization of the lamination on which $|du| = L$. This set is geometrically determined and must be the same for any two $\infty$-harmonic maps $u$ and $u'$  (see Corollary~\ref{strsuppp}). Moreover
%$du = du'$ on the support of any limiting measure $dv$.

The uniqueness proofs do not carry over for maps into $S^1$.  They are based on constructions which involve taking the maximum of $u$. This has nothing to do with the hyperbolic metric. One would meet the same problems in the following problem in Euclidean space.
Let $\Omega$ be an annular region in $\R^2$, choose a map $u_0: \Omega \rightarrow S^1$, let $b =u_0 \big |_{\partial \Omega}$.  Find the $\infty$-harmonic map $u$ with $u \big |_{\partial \Omega}=b$.   Existence and regularity are straightforward. Is $u$ unique?

\begin{conjecture}\label{Conjecture 3}  The  BV section $v: M \rightarrow L$ in Theorem~\ref{transmeasure}  is unique and  the limit v is equally distributed.
\end{conjecture}

 This is two problems. The first part is to show the cohomology class of $L$ is unique, and the second is the analogous of Conjecture~\ref{Conjecture 2} for $v$ instead of $u$. Recall from Theorem~\ref{thm:equivsit} that given a cohomology class $\rho \in H^1(M, \R)$ there is an associated cohomology class $\alpha \in  H^1(M, \R)$ representing the cohomology class of $L$. It can be seen that $\alpha$ is unique and the map $\rho \mapsto \alpha$ is a well defined map $H^1(M, \R) \rightarrow H^1(M, \R)$ (cf. \cite{daskal-uhlen3}).
  
  The second part is that, given $\alpha$ (or equivalently $L$), the least gradient map $v: M \rightarrow L$ is unique. This is a serious deficiency, since $v$ determines a transverse measure, and a lamination which is not connected has many transverse measures. We are conjecturing that we obtain the measure which gives equal weight to the components. 
  The best way to think of this is the case where the maximum stretch lamination consists of a finite number of closed geodesics $\gamma_ j$ all of which are parameterized locally with $\tilde u(\gamma_j (t)) = Lt$ where
$0 \leq  t \leq n(j)/L$. We then conjecture these geodesics have jumps $\delta/n(j)$. This means that, as you pass around the fiber, the jumps on each geodesic of $v$ are equally distributed. As we saw, the non-closed isolated leaves of the lamination have no jumps, but the idea can be extended to the Cantor components which have no isolated leaves.

\begin{problem}\label{Problem 5}
 We know by Theorem~\ref{thmlegr} that any weak limit $v$ of the $q$-harmonic functions $v_q$ is of least gradient. Use the map $v$ to prove directly  that the support of $dv$ is a geodesic lamination, bypassing the need to use any properties of $\infty$-harmonic functions and our proof that the best Lipschitz constant is achieved on a geodesic lamination. 
  \end{problem}

  A partial converse to the statement of Problem~\ref{Problem 5} should also hold:

  \begin{conjecture} Suppose that $\lambda$ is an arbitrary oriented lamination on $M$ with a transverse measure.  Let $\tilde v: \tilde M \rightarrow \R$ be a primitive for the Ruelle Sullivan current associated with the transverse measure. For any ball $B \subset \tilde M$, 
  $\tilde v \big|_B$ is of least gradient.
  \end{conjecture}
  
  The main point in the above statement is that the boundary of the sets $\tilde v \geq t$ in $B$ are geodesics. From this, one should be able to deduce like in the Euclidean case that $\tilde v$ is a locally a map of least gradient. 

\begin{thmconj}\label{Theorem-Conjecturepun} The results of this paper extend to surfaces with punctures. 
\end{thmconj}
Throughout the paper we restricted ourselves to the case of closed manifolds. However, Thurston's theory works also for laminations on surfaces with punctures. Most of the results in this paper are local and  carry through also for punctured surfaces without significant change.

The next problem is the analogue of Thurston's construction \cite{thurston} adapted to our situation. Before we state the problem we need some notation. Given a cohomology class $\rho \in H^1(M,\R)$ and  a hyperbolic metric $g$ on $M$ we can consider $K$ as in (\ref{normlength}) defined on the space of measured laminations $\mathcal M \mathcal L$. Equivalently, let $\mathcal M \mathcal L_\rho$ denote the space of measured geodesic laminations whose homology class is dual to $\rho$, i.e measured laminations $\lambda$ subject to the topological constraint $\rho(\lambda)=1$. On this space we consider the length functional
\[
l_g: \mathcal M \mathcal L_\rho \rightarrow \R
\] 
associating to a measured lamination $\lambda$ its length $l_g(\lambda)$ with respect to $g$.
\begin{problem}\label{toyprobm} Compute critical points of the  function $l_g$ and study the connection with $\rho$-equivariant best Lipschitz functions and their maximum stretch laminations. Carry through Thurston's construction for this case. 
\end{problem}

This should be a very doable problem.
%\begin{problem}\label{Conjecture 6} Given a line bundle $E$ over $ M$,  there exists an $\infty$-harmonic section $u: M\rightarrow E$ and a dual least gradient section $v: M \rightarrow L$ for a line bundle $L$.
%\end{problem}
%%This conjecture bears connection with Thurston's work \cite{thurston} and we will come back later.
% Prescribing the line bundle $E$ for $u: M\rightarrow E$ is an analogue of the Dirichlet problem for $u$. Prescribing the bundle $v: M \rightarrow L$ is an analogue of a Neuman problem for $u$. For finite $1< q \leq p < \infty$, both Dirichlet and Neumann problems have solutions (unique up to equivalence). We are asking the question for $p = \infty$, $q = 1$.
%\begin{conjecture}\label{Conjecture 7}
% A transverse measure to a geodesic lamination determines a singular quadratic differential. Formulate this properly and describe a least gradient quadratic differential.
%\end{conjecture}

\begin{problem}\label{Problem 8} The gradient field for a p harmonic function $u_p$ determines an interval exchange map on a regular fiber $u_p^{-1}(t)$. Study the invariants of these as $p \rightarrow \infty$.
 \end{problem}
 We refer the reader to the very readable paper of Masur \cite{masur}.

\begin{problem}\label{Conjecture 10} Investigate the theory of $p$-harmonic maps, $\infty$-harmonic maps and least gradient maps into trees and their duality. 
  \end{problem}
%  Some aspects of the theory of  \cite{gromov-schoen}, \cite{korevaar-schoen1} and \cite{korevaar-schoen2} go through without much trouble. By utilizing the notion of energy of maps to trees developed in \cite{gromov-schoen}, \cite{korevaar-schoen1} and \cite{korevaar-schoen2}   the analysis  should carry  over to this case.
For a  combinatorial approach to best Lipschitz maps to trees, see \cite{naor-sheffield}. 

\begin{problem}\label{Problem 12}
 Develop a theory of $\infty$-harmonic maps $u: M^3 \rightarrow S^1$ where $M^3$ is a hyperbolic 3-manifold.
 \end{problem}

From Section~\ref{sect:crandal} we have shown that the set of maximum stretch $L_u = L$ is a geodesic lamination. However, the geometry of the dual problems or the two form
$dv_q =|du_p|^{p-2}*du_p$ and the limit $q \rightarrow 1$ is unexplored territory. The dual 2-form $*du$ is a transverse area measurement which is far less rigid than length. Purely geometric descriptions of hyperbolic 3-manifolds which fiber over a circle are sorely lacking, so it is worth exploring any possibility.

\begin{problem}\label{Problem 15} Let $M = M^3$ be a hyperbolic manifold which fibers over a circle. Study least gradient maps $v: M \rightarrow S^1$, or more generally equivariant  least gradient maps to $\R$ or trees.
 \end{problem}
 This is a promising problem, since a lot is known about least gradient maps in three dimensions. For the Dirichlet problem for domains in $\R^3$, the level sets of a least gradient function are minimal surfaces; hence we expect this least gradient map to tie into the theory of minimal surfaces in $M$. 
 
 As mentioned already the motivation for this paper was in understanding Thurston's work of best Lipschitz maps between surfaces. As such, the results of this paper only serve as a toy problem.  We conclude by stating the motivating problem and a quick preview of our approach in the forthcoming papers \cite{daskal-uhlen2} and \cite{daskal-uhlen3}:
 
 \begin{problem}\label{Problem 11} Is there an analogous analytical theory of $\infty$-harmonic maps $u: M \rightarrow N$ between hyperbolic surfaces which ties into Thurston's results on the asymmetric metric on Teichm\"uller space using best Lipschitz maps?
\end{problem}
 
 The analysis is entirely lacking for this problem, although there is a  topological theory due to Thurston and his school (cf. \cite{thurston}, \cite{papa} and \cite{kassel}). 
  The main problem in the analysis is the lack of a good notion  of viscosity solutions for systems and this seems out of reach at this point. In a series of follow-up articles we will  bypass this issue and develop a theory analogous to this paper that ties in with Thurston.
  
  Like in this paper, the first step is to define a good notion of $p$-approximations of best Lipschitz maps.  In the case when the target has dimension greater than 1, $p$-harmonic maps is not the right notion since they do not converge to best Lipschitz maps.  We consider  maps minimizing the $p$-Schatten-von Neumann norm of the gradient instead of the $L^p$-norm.  This version of $p$-harmonic maps have even weaker regularity properties and don't satisfy maximum principle. This makes it hard to prove  comparison with cones. We have to
  rely on the result of Gueritaud-Kassel \cite{kassel} in order to show that the maximum stretch set of the infinity harmonic map contains  Thurston's canonical lamination. 
  
  Another difference  with the scalar case is the construction of the dual functions and the limiting measures. When the target is a hyperbolic surface,  $v$ has values in the Lie algebra of $SO(2,1)$ instead of $\R$. We construct these measures by analyzing the conservation laws coming from the symmetries of the target and extend the support argument  to show that $dv$ has support on the canonical lamination.
  
  There are two points that we entirely missed in this paper which we will explore in \cite{daskal-uhlen2} and \cite{daskal-uhlen3}. The first is the role of symmetries of the domain manifold. Much like $dv$, there exist  Radon measures associated to  best Lipschitz maps corresponding to the symmetries of the domain. Again the support of this measure is on the canonical lamination. The second point is the interpretation of the cohomology class of $dv$ as well as its counterpart coming from the symmetries of the domain in terms of the first variation of the Lipschitz constant.

\end{document}